\documentclass[12pt,reqno]{amsart}

\usepackage{amssymb,amsmath,graphicx,
	amsfonts,euscript}
\usepackage{color}
\usepackage{mathrsfs}
\allowdisplaybreaks

\theoremstyle{plain}

\numberwithin{equation}{section}
\newtheorem{theorem}{Theorem}[section]
\newtheorem{proposition}[theorem]{Proposition}
\newtheorem{lemma}[equation]{Lemma}

\newtheorem{definition}[equation]{Definition}

\numberwithin{equation}{section}

\theoremstyle{remark}
\newtheorem{remark}[theorem]{Remark}

\newcommand{\R}{{\mathbb R}}
\def\div{ \hbox{\rm div}\,  }
\newcommand\Z{{\mathbb{Z}}}
\def\p{{\mathcal P}}
\def\q{{\mathcal Q}}
\def\ddj{\dot \Delta_j}
\def\f{\frac}
\def\e{\mathcal{E}_\infty(t)}
\def\lon{\varepsilon}
\def\er{\mathcal{E}_1(t)}
\def\ga{\Gamma}
\def\u{\mathbf{u}}

\begin{document}

\title[Global solutions and time decay rates to  Oldroyd-B model ]{Global  wellposedness and large time behavior of solutions to the  $N$-dimensional compressible\\ Oldroyd-B model}

\author[Xiaping Zhai,  and Yongsheng Li]{Xiaoping Zhai$^\dag$ and Yongsheng Li$^\ddag$}

\address{$^{\dag}$ School  of Mathematics and Statistics, Shenzhen University,
 Shenzhen, 518060, China}

\email{pingxiaozhai@163.com}

\address{$^\ddag$ School of Mathematics,
South China University of Technology,
Guangzhou, 510640, China}

\email{yshli@scut.edu.cn}

\vskip .2in
\begin{abstract}
	The purpose of this work is to study the global wellposedness and large time behavior results of
strong solutions for the compressible Oldroyd-B model derived by Barrett,  Lu,  S\"uli (Commun. Math. Sci., 15, 1265--1323, 2017).
Exploiting the Harmonic analysis tools (especially Littlewood-Paley
theory), we first  study the global well-posedness of the model with small initial data in spaces with low regularity.
Then,
under  a suitable  condition involving only the low frequency of the  initial data, we also obtain the optimal decay rates of the solutions. Compared with the result by Wang and Wen (Math. Models Methods Appl. Sci., 30, 139--179, 2020), the polymer number density is allowed to vanish and the stress tensor isnear zero equilibrium.

	\end{abstract}
\maketitle
%\noindent {\bf Key Words:}
%{Compressible Oldroyd-B model; Global  solutions; Decay rates; Besov spaces}

%\noindent {\bf Mathematics Subject Classification (2010)} {35Q53,76N10, 74H40}

\section{ Introduction and the main results }
In this paper, we mainly consider the Cauchy problem of the following
 compressible Oldroyd-B model:
\begin{eqnarray}\label{sys}
\left\{\begin{aligned}
&\partial_t\rho+\div(\rho \u)=0,\\
&\partial_t\eta+\div(\eta \u)-\lon\Delta\eta=0,\quad\quad x\in \R^n, \quad t>0,\\
&\partial_t{\mathbb{T}}+(\u\cdot\nabla){\mathbb{T}}+ {\mathbb{T}} \div \u  -
(\nabla \u{\mathbb{T}}+{\mathbb{T}}\nabla^{\top} \u)-\lon\Delta{\mathbb{T}}= \frac{\kappa\,A_0}{2\lambda_1}\eta  \,\mathbb{Id} - \frac{A_0}{2\lambda_1}{\mathbb{T}} ,\\
&\rho(\partial_t\u+\u\cdot\nabla \u)-\mu\Delta \u-(\lambda+\mu)\nabla\div \u+\nabla
P=\div({\mathbb{T}} - (\kappa L\eta + \zeta\, \eta^2)\,\mathbb{Id}\, \big),
\end{aligned}\right.
\end{eqnarray}
for $(t,x)\in \mathbb{R}_+\times\mathbb{R}^n\,(n= 2,3)$. Here $\rho=\rho(t,x)\in \mathbb{R}_+$ is the density function of the fluid, ${ \u}={ \u}(t,x)\in \mathbb{R}^n$ is the velocity. The symmetric matrix function ${\mathbb{T}} = ({\mathbb{T}}_{i,j})$, $1\leq i,j \leq n$ is the extra stress tensor
and ${\eta}={\eta}(t,x)\in\mathbb{R}_+$ represents the polymer number density defined as the
integral of a probability density function $\psi$  with respect to the conformation vector, which is a microscopic variable  in the modeling of dilute polymer chains, i.e.,

$$\eta=\int_{\R^n}\psi(t,x,q)\,dq\ge 0.$$
Here $\psi$ is governed by the Fokker-Plank equation.

 The viscosities constant $\mu$ and $\lambda$ are supposed to  satisfy $\mu>0$ and $n\lambda+2\mu\ge 0$.
 In particular, the parameters $\kappa$,  $\lon$,  $A_0$, $\lambda_1$ are all positive numbers, whereas $\zeta \geq 0$ and $L\geq 0$ with $\zeta + L \neq 0$. The term $\kappa L \eta + \zeta \eta^2$ in the momentum equation \eqref{sys} can be seen as the {polymer pressure}, compared to the fluid pressure $P(\rho)=A\rho^\gamma$.
 System \eqref{sys} is supplemented
with the initial data
\begin{equation}\label{initial}
  (\rho, \u, \eta,{\mathbb{T}} )|_{t=0}=(\rho_0(x), \u_0(x), \eta_0(x), {\mathbb{T}}_0(x)), \,\,x\in \mathbb{R}^n,
\end{equation}
and with far field behaviors
\begin{equation}\label{far}
\rho\to\bar{\rho},\quad\quad \u\to\mathbf{0},\quad\quad \eta\to 0,\quad\quad {\mathbb{T}}\to 0\mathbb{Id}\quad \mathrm{as}
\quad|x|\to \infty.
\end{equation}

Micro-macro models of dilute polymeric fluids that arise from statistical physics are based on coupling the Navier-
Stokes system to the Fokker-Planck equation. In these models polymer molecules are idealized as chains of massless
beads, linearly connected with inextensible rods or elastic springs.
The model we consider here was first derived by Barrett, Lu, and S\"uli \cite{model} via micro-macro analysis
of the compressible Navier-Stokes-Fokker-Planck system
studied in a series of papers by Barrett and S\"uli \cite{Barrett1}--\cite{Barrett5}.
Barrett and S\"uli  \cite{model} obtained the global-in-time finite-energy weak solutions with
large initial data in $\R^2$. However, the uniqueness of the
global weak solution is still open. Later, Lu and
Zhang \cite{luyong} proved the local wellposedness, weak-strong uniqueness and a refined blow-up criterion involving only the upper bound of the fluid density.  Wang and Wen \cite{wenhuanyao} obtained the global wellposedness of \eqref{sys} as well as associated time-decay estimates in Sobolev space, if the initial data is near a nonzero equillbrium state.
 Most recently, the first author of the present paper in \cite{zhaixiaopingyuxie1}  justify  the  low Mach number convergence to the incompressible
Oldoryd-B model   for viscous compressible Oldoryd-B  model   in the {ill-prepared data} case.
When neglecting the stress diffusion in \eqref{sys} and assuming further the extra stress tensor is a scalar matrix,
Lu and Pokorn\'y  \cite{luyong2} obtained the  global weak solutions with large data.
%, Zhai \cite{zhaixiaopingyuxie2} obtained the  global strong solutions with small initial data.

 It is interesting to note that the model \eqref{sys} with $\eta=0$ is  related to the following compressible Oldoryd-B model, i.e.,
\begin{eqnarray}\label{sys1}
\left\{\begin{aligned}
&\partial_t\rho+\div(\rho \u)=0,\\
&\rho(\partial_t\u+(\u\cdot\nabla) \u)-\mu\Delta \u-(\lambda+\mu)\nabla\div \u+\nabla
P=\mu_1\div{\mathbb{T}},\\
&\partial_t{\mathbb{T}}+(\u\cdot\nabla){\mathbb{T}}+g({\mathbb{T}},
\nabla \u)+\beta{\mathbb{T}}=\mu_2 D(\u),\quad\quad\quad\quad x\in \R^n, \quad t>0,
\end{aligned}\right.
\end{eqnarray}
where $D(\u)$ is the symmetric part of $\nabla \u$, and $\Omega(\u)$ is the skew-symmetric part of $\nabla \u$, namely
\begin{equation*}%\label{omegau}
D(\u) = \frac{1}{2} \big( \nabla \u + (\nabla \u)^{\top} \big),\quad \Omega(\u) = \frac{1}{2} \big( \nabla \u - (\nabla \u)^{\top} \big),
\end{equation*}
 and
$$g({\mathbb{T}}, \nabla \u)\stackrel{\mathrm{def}}{=}{\mathbb{T}}
\Omega(\u)-\Omega(\u){\mathbb{T}}-b\left(D(\u){\mathbb{T}}+{\mathbb{T}} D(\u)\right),\quad\hbox{$b$ is a parameter in $[-1,1]$}.$$

The  Oldroyd--B model  attracts continuous attentions of mathematicians, however,
there has few  known results concerning compressible Oldroyd-B models of \eqref{sys1}.
 Lei \cite{leizhen2006} and Gullop\'{e} {\it et al.} \cite {GST} studied
the incompressible limit problem of the compressible Oldroyd-B model in a
torus and bounded domain
of $\R^3$, respectively.
Recently,
Zi \cite{zuiruizhao2017} obtained the global small solutions of \eqref{sys1} in the critical $L^2$ Besov spaces. The first author in the present paper and Chen  \cite{zhaixiaopingarxiv} generalized the result of \cite{zuiruizhao2017} about \eqref{sys1}    without damping mechanism to the critical $L^p$  spaces.
For the compressible Oldroyd type model
based on the deformation tensor, see the results  \cite{huxianpeng}, \cite{huxianpeng2013}, \cite{pan2019dcdsa}, \cite{qian} and references therein.
Let $\rho$ be constant, the system \eqref{sys1} reduces to be the incompressible Oldroyd--B model, which has made rather rich results, see \cite{chemin}, \cite{ER},  \cite{GS}, \cite{GS2}, \cite{lin2012}, \cite{Renardy}.

Let us give now more details on the form of the solutions that we are going to consider.
Let  $\bar{\rho}=1$ in \eqref{far} and define
$\rho=1+a$,  we can reformulate the system \eqref{sys} into the following form:
\begin{eqnarray}\label{m}
\left\{\begin{aligned}
&\partial_ta+\div \u=-\div(a\u),\\
&\partial_t\eta-\lon\Delta\eta=-\div(\eta \u),\\
&\partial_t{\mathbb{T}}+ \frac{A_0}{2\lambda_1}{\mathbb{T}}+(\u\cdot\nabla){\mathbb{T}}-\lon\Delta{\mathbb{T}}= \frac{\kappa\,A_0}{2\lambda_1}\eta  \,\mathbb{Id}+F({\mathbb{T}}, \u),\\
&\partial_t\u+\u\cdot\nabla \u-\mu\Delta \u-(\lambda+\mu)\nabla\div \u+\nabla
a=\div{\mathbb{T}}-\kappa L\nabla\eta +G(a,\u,\eta,{\mathbb{T}}),\\
&(a, \u, \eta,{\mathbb{T}} )|_{t=0}=(a_0(x), \u_0(x), \eta_0(x), {\mathbb{T}}_0(x)),
\end{aligned}\right.
\end{eqnarray}
with
\begin{align*}%\label{}
\nu\stackrel{\mathrm{def}}{=}\lambda+2\mu,\quad
 I(a)\stackrel{\mathrm{def}}{=}\frac{a}{1+a},\quad k(a)\stackrel{\mathrm{def}}{=}
-\frac{P'(1+a)}{1+a}+P'(1)\quad  \hbox{with $P'(1)=1$},
\end{align*}
\begin{align*}%\label{}
F({\mathbb{T}}, \u)\stackrel{\mathrm{def}}{=} &
(\nabla \u\ {\mathbb{T}}+{\mathbb{T}}\nabla^\top \u)- {\mathbb{T}}\div \u,\\
G(a,\u,\eta,{\mathbb{T}})\stackrel{\mathrm{def}}{=} &k(a)\nabla a-I(a)(\mu\Delta \u+(\lambda+\mu)\nabla\div \u)- I(a)( \div{\mathbb{T}}- \nabla\eta)-\zeta(1-I(a))\eta\nabla\eta.
\end{align*}

The first main result of the paper is stated as follows.
\begin{theorem}(Local wellposedness)\quad \label{dingli1}
Let   $ n=2,3$ and $1<p<2n$. For any  $\u_0\in \dot{B}_{p,1}^{\frac np-1}(\R^n)$,  $(\eta_0,{\mathbb{T}}_0)\in \dot{B}_{p,1}^{\frac np}(\R^n)$ and $a_0\in \dot{B}_{p,1}^{\frac np}(\R^n)$  with $1 + a_0$ bounded away from zero. Then there exists a positive time
$T$ such that the system \eqref{m} has a unique solution with
\begin{align*}
&a\in C_b([0,T ];{\dot{B}}_{p,1}^{\frac {n}{p}}),\quad \u\in C_b([0,T ];{\dot{B}}_{p,1}^{\frac {n}{p}-1})\cap L^{1}
([0,T];{\dot{B}}_{p,1}^{\frac np+1}),\\
& (\eta,{\mathbb{T}})\in C_b([0,T ];{\dot{B}}_{p,1}^{\frac {n}{p}})\cap L^{1}
([0,T];{\dot{B}}_{p,1}^{\frac np+2}),
\quad {\mathbb{T}} \in L^{1}
([0,T];{\dot{B}}_{p,1}^{\frac np}).
\end{align*}
\end{theorem}

\begin{remark}
 The  solutions constructed here allow
    the
regularity exponent $n/p-1$ for the velocity becomes negative.
Our result thus applies  to {large} highly oscillating initial velocities (see  \cite{ chenqionglei}
for more explanation).
\end{remark}
Before presenting the second result of the paper, we give the following notation.
 Let $\mathcal{S}(\R^n)$ be the space of
rapidly decreasing functions over $\R^n$ and $\mathcal{S}'(\R^n)$ its dual
space. For any $z \in\mathcal{S}'(\R^n)$,
the low and high frequency parts are expressed as
\begin{align*}
z^\ell\stackrel{\mathrm{def}}{=}\sum_{j\leq j_0}\ddj z\quad\hbox{and}\quad
z^h\stackrel{\mathrm{def}}{=}\sum_{j>j_0}\ddj z
\end{align*}
for some fixed   integer $j_0\ge 1$ (the value of which follows from the proof of the main theorems). %depending only on $p$.
The corresponding  truncated semi-norms are defined  as follows:
\begin{align*}\|z\|^{\ell}_{\dot B^{s}_{p,r}}
\stackrel{\mathrm{def}}{=}  \|z^{\ell}\|_{\dot B^{s}_{p,r}}
\ \hbox{ and }\   \|z\|^{h}_{\dot B^{s}_{p,r}}
\stackrel{\mathrm{def}}{=}  \|z^{h}\|_{\dot B^{s}_{p,r}}.
\end{align*}
Denote $$
\Lambda\stackrel{\mathrm{def}}{=}\sqrt{-\Delta}, \quad\hbox{and}\quad \p=\mathcal{I}-\mathcal{Q}\stackrel{\mathrm{def}}{=}\mathcal{I}-\nabla\Delta^{-1}\div.$$

The second main result of the paper is stated as follows.
\begin{theorem}(Global  wellposedness)\quad \label{dingli2}
Let   $ n=2,3$ and
\begin{equation*}
2\leq p \leq \min(4,{2n}/({n-2}))\quad\hbox{and, additionally, }\  p\not=4\ \hbox{ if }\ n=2.
\end{equation*}
 For any $(a_0^\ell,\u_0^\ell)\in \dot{B}_{2,1}^{\frac n2-1}(\R^n)$,  $\eta_0^\ell\in \dot{B}_{2,1}^{\frac n2-2}(\R^n)$,  ${\mathbb{T}}_0^\ell\in \dot{B}_{2,1}^{\frac n2}(\R^n)$,  and  $(a^h_0,{\mathbb{T}}_0^h)\in \dot{B}_{p,1}^{\frac np}(\R^n)$, $(\u_0^h,\eta_0^h)\in \dot{B}_{p,1}^{\frac np-1}(\R^n)$.
 There
exists a positive constant $c_0$ such that if,
\begin{align}\label{smallness}
\|(a^\ell_0,\u^\ell_0)\|_{\dot{B}_{2,1}^{\frac {n}{2}-1}}+\|{\mathbb{T}}_0^\ell\|_{\dot{B}_{2,1}^{\frac {n}{2}}}+\|\eta_0^\ell\|_{\dot{B}_{2,1}^{\frac {n}{2}-2}}+\|(a^h_0,{\mathbb{T}}^h_0)\|_{\dot B^{\frac  np}_{p,1}}+\| (\u_0^h,\eta_0^h)\|_{\dot B^{\frac  np-1}_{p,1}}\leq c_0,
\end{align}
then
the system \eqref{m} has a unique global solution $(a,u,\eta,{\mathbb{T}})$ so that
\begin{align*}%\label{}
&a^\ell\in C_b(\R^+;{\dot{B}}_{2,1}^{\frac {n}{2}-1})\cap L^{1}
(\R^+;{\dot{B}}_{2,1}^{\frac n2+1}),\quad
 a^h\in C_b(\R^+;{\dot{B}}_{p,1}^{\frac {n}{p}})\cap L^{1}
(\R^+;{\dot{B}}_{p,1}^{\frac np}),\\
&\u^\ell\in C_b(\R^+;{\dot{B}}_{2,1}^{\frac {n}{2}-1}\cap L^{1}
(\R^+;{\dot{B}}_{2,1}^{\frac n2+1}),\quad \u^h\in C_b(\R^+;{\dot{B}}_{p,1}^{\frac {n}{p}-1}\cap L^{1}
(\R^+;{\dot{B}}_{p,1}^{\frac np+1}),\\
&\eta^\ell\in C_b(\R^+;{\dot{B}}_{2,1}^{\frac {n}{2}-2})\cap L^{1}
(\R^+;{\dot{B}}_{2,1}^{\frac n2}),\quad \eta^h\in C_b(\R^+;{\dot{B}}_{p,1}^{\frac {n}{p}-1})\cap L^{1}
(\R^+;{\dot{B}}_{p,1}^{\frac np+1}),\\
&{\mathbb{T}}^\ell\in C_b(\R^+;\dot{B}_{2,1}^{\frac n2})\cap L^{1}
(\R^+;{\dot{B}}_{2,1}^{\frac n2}),\quad {\mathbb{T}}^h\in C_b(\R^+;\dot{B}_{p,1}^{\frac np})\cap L^{1}
(\R^+;{\dot{B}}_{p,1}^{\frac np+2}).
\end{align*}
Moreover,  there exists some constant $C$ such that
\begin{align}\label{xiaonorm}
&X(t)\leq Cc_0,
\end{align}
with
\begin{align*}%\label{yi5}
X(t)\stackrel{\mathrm{def}}{=}&\|(a,\u)\|^\ell_{\widetilde{L}^\infty_t(\dot{B}_{2,1}^{\frac{n}{2}-1})}
+\|\eta\|^\ell_{\widetilde{L}^\infty_t(\dot{B}_{2,1}^{\frac{n}{2}-2})}
+\|{\mathbb{T}}\|^\ell_{\widetilde{L}^\infty_t(\dot{B}_{2,1}^{\frac{n}{2}})}
+\|(\u,\eta)\|^h_{\widetilde{L}^\infty_t(\dot{B}_{p,1}^{\frac{n}{p}-1})}
\nonumber\\
&\quad+\|(a,{\mathbb{T}})\|^h_{\widetilde{L}^\infty_t(\dot{B}_{p,1}^{\frac{n}{p}})}+\|(a,\u)\|^\ell_{L^1_t(\dot{B}_{2,1}^{\frac{n}{2}+1})}+\|(\eta,{\mathbb{T}})\|^\ell_{L^1_t(\dot{B}_{2,1}^{\frac{n}{2}})}+ \|{\mathbb{T}}^\ell\|_{L^1_t(\dot B^{\frac  n2+2}_{2,1})}\nonumber\\
&\quad\quad+\|(a,{\mathbb{T}})\|^h_{L^1_t(\dot{B}_{p,1}^{\frac{n}{p}})}+\|(\u,\eta)\|^h_{L^1_t(\dot{B}_{p,1}^{\frac{n}{p}+1})}
+\|{\mathbb{T}}\|^h_{L^1_t(\dot{B}_{p,1}^{\frac{n}{p}+2})}.
\end{align*}

\end{theorem}
\begin{remark}
Compared with the  result by  Wang and  Wen \cite{wenhuanyao}, the polymer number density is allowed to vanish and the stress tensor
can   be near zero equilibrium here.
\end{remark}

With the global solutions constructed above,
next, a natural problem is what is the large time asymptotic behavior of this solutions.
The study of the large-time behavior of solutions to the partial different equations is also an old subject. We
refer for instance to \cite{xujiang}, \cite{xujiang2019arxiv},  \cite{zhaixiaoping} for the compressible Navier-Stokes equations and \cite{pan2019dcdsa}, \cite{wenhuanyao},  \cite{zhaixiaopingarxiv} for Oldroyd-B model.

One can now state the third result of the present paper.
\begin{theorem}$\mathrm{(}$Optimal decay$\mathrm{)}$\quad\label{dingli3}
Let ~$(a,\u,\eta,{\mathbb{T}})$ be the global small solutions addressed by Theorem \ref{dingli2}. For any $\frac{n}{2}-\frac{2n}{p}\le\sigma<\frac{n}{2}-1,$
and  $(a_0^\ell,\u_0^\ell)\in{\dot{B}_{2,\infty}^{\sigma}}(\R^n)$, $\eta_0^\ell\in{\dot{B}_{2,\infty}^{\sigma-1}}(\R^n)$, ${\mathbb{T}}_0^\ell\in{\dot{B}_{2,\infty}^{\sigma+1}}(\R^n)$, we have the following time-decay rate
\begin{align*}
\|\Lambda^{\gamma_1} (a,\u)\|_{L^p}
\le C(1+t)^{-\frac{n}{2}(\frac 12-\frac 1p)-\frac{\gamma_1-\sigma}{2}},\quad\quad &\forall \gamma_1\in\left(\frac np-\frac n2+\sigma,\frac np-1\right],\\
\|\Lambda^{\gamma_2}\eta\|_{L^p}
\le C(1+t)^{-\frac{n}{2}(\frac 12-\frac 1p)-\frac{\gamma_2-\sigma+1}{2}},\quad\quad &\forall \gamma_2\in\left(\frac np-\frac n2-1+\sigma,\frac np-1\right],\\
\|\Lambda^{\gamma_3}{\mathbb{T}}\|_{L^p}
\le C(1+t)^{-\frac{n}{2}(\frac 12-\frac 1p)-\frac{\gamma_3-\sigma-1}{2}},\quad\quad &\forall \gamma_3\in\left(\frac np-\frac n2+1+\sigma,\frac np\right].
\end{align*}
\end{theorem}

\begin{remark}\label{yaya3}
Let $p=2$, one can deduce from above decay estimates  that
\begin{align*}
&\|\Lambda^{\gamma_1} (a,\u)\|_{L^2} \le C (1+t)^{-\frac {\gamma_1}{2}+\frac {\sigma}{2}},\\
&\|\Lambda^{\gamma_2}\eta\|_{L^2}\le C (1+t)^{-\frac {\gamma_2}{2}+\frac {\sigma-1}{2}},\\
&\|\Lambda^{\gamma_3}{\mathbb{T}}\|_{L^2}\le C (1+t)^{-\frac {\gamma_3}{2}+\frac {\sigma+1}{2}},
\end{align*}
which coincides with the  heat flows, thus our decay rate is optimal in some sense.
\end{remark}

\subsection*{\bf Scheme of the proof and organization of the paper.}
The equations of the polymer number density $\eta$ and the extra stress tensor ${\mathbb{T}}$ in \eqref{m} are two heat-type flow, moreover,
when the polymer number density $\eta$ and the extra stress tensor ${\mathbb{T}}$ vanish, the system \eqref{m} reduces to the
barotropic Navier-Stokes equations, thus, we can modify the method used  in  \cite{danchin2014} and  \cite{haspot}   to
establish the local wellposedness.
   As the process is lengthy but standard, we only sketch some details in Section 3 to complete the proof of Theorem \ref{dingli1}.

 To prove the Theorem \ref{dingli2} regarding the global solutions with small initial data,
we use the bootstrap argument, which consists of two main steps.
The first step is to establish the {\it a priori} bounds while the second is to apply
and complete the bootstrap argument by using the {\it a priori} bounds. We set the argument in Section 4 which will further be divided into three subsections.
 Main efforts are
devoted to obtaining suitable {\it a priori} bounds, and we put it in the first two subsections.  To do so,  we separate the low frequency
 from the high frequency to distinguish their different behaviors.

We shall prove the Theorem \ref{dingli3} in Section 5. Inspired by the papers \cite{xujiang2019arxiv} and \cite{zhaixiaoping},
 our main task is to establish a Lyapunov-type inequality in time for energy norms
(see \eqref{sa58}) by using the pure energy argument (independent of spectral analysis).

 In Section 2, we recall the Littlewood-Paley theory and give some useful lemmas about product laws, commutators estimates in Besov spaces.  In the Appendix A, we give some new product laws in Besov spaces.

Let us complete this section by describing the notations which will be used in the sequel.

\noindent{\rm\bf Notations: }
For two operators $A$ and $B$, we denote $[A,B]=AB-BA$, the commutator between $A$ and $B$.
The letter $C$ stands for a generic constant whose meaning is clear from the context.
We  write $a\lesssim b$ instead of $a\leq Cb$.
Given a Banach space $X$, we shall denote $\|(a,b)\|_{X}=\|a\|_{X}+\|b\|_{X}$.

For $X$ a Banach space and $I$ an interval of $\mathbb{R}$, we denote by $C(I;X)$ the set of continuous functions on $I$ with values in $X$,
and by $C_{b}(I;X)$ the subset of bounded functions of $C(I;X)$.
For $q\in [1,+\infty]$, $L^q(I;X)$ stands for the set of measurable functions on $I$ with values in $X$, such that $t\mapsto\|f(t)\|_{X}$ belongs to $L^q(I)$.
For short, we  write $L_T^q(X)$ instead of $L^q((0,T);X)$. We always let $(d_j)_{j\in\mathbb{Z}}$ be a
generic element of ${\ell}^1(\mathbb{Z})$  so that $\sum_{j\in\mathbb{Z}}d_j=1$.

\section{ Preliminaries }
For  readers' convenience, in this section, we list some basic knowledge on  Littlewood-Paley  theory.
The {Littlewood-Paley decomposition}  plays a central role in our analysis.
To define it,   fix some  smooth radial non increasing function $\chi$
supported in the ball $B(0,\frac 43)$ of $\R^n,$ and with value $1$ on, say,   $B(0,\frac34),$ then set
$\varphi(\xi)=\chi(\frac{\xi}{2})-\chi(\xi).$ We have
$$
\qquad\sum_{j\in\Z}\varphi(2^{-j}\cdot)=1\ \hbox{ in }\ \R^n\setminus\{0\}
\quad\hbox{and}\quad \mathrm{Supp}\,\varphi\subset \Big\{\xi\in\R^n : \frac34\leq|\xi|\leq\frac83\Big\}\cdotp
$$
The homogeneous dyadic blocks $\dot{\Delta}_j$ are defined on tempered distributions by
$$\dot{\Delta}_j u\stackrel{\mathrm{def}}{=}\varphi(2^{-j}D)u\stackrel{\mathrm{def}}{=}{\mathcal F}^{-1}(\varphi(2^{-j}\cdot){\mathcal F} u).
$$
Let us remark that, for any homogeneous function $A$ of order 0 smooth outside 0, we have
\begin{equation*}\label{}
\forall p\in[1,\infty],\quad\quad\|\ddj (A(D) u)\|_{L^p}\le C\|\ddj u\|_{L^p}.
\end{equation*}
\begin{definition}
Let $p,r$ be in~$[1,+\infty]$ and~$s$ in~$\R$, $u\in\mathcal{S}'(\R^n)$. We define the Besov norm by
$$
\|u\|_{{\dot{B}^s_{p,r}}}\stackrel{\mathrm{def}}{=}\big\|\big(2^{js}\|\ddj
u\|_{L^{p}}\big)_j\bigr\|_{\ell ^{r}({\mathop{\mathbb Z\kern 0pt}\nolimits})}.
$$
We then define the spaces
$\dot{B}_{p,r}^s\stackrel{\mathrm{def}}{=}\left\{u\in\mathcal{S}'_h(\R^n),\big|
\|u\|_{\dot{B}_{p,r}^s}<\infty\right\}$, where $u\in \mathcal{S}'_h(\R^n)$ means that $u\in \mathcal{S}'(\R^n)$ and $\lim_{j\to -\infty}\|\dot{S}_ju\|_{L^\infty}=0$ (see Definition 1.26 of \cite{bcd}).
\end{definition}

When employing parabolic estimates in Besov spaces, it is somehow natural to take the time-Lebesgue norm before performing the summation for computing the Besov norm. So we next introduce the following Besov-Chemin-Lerner space $\widetilde{L}_T^q(\dot{B}_{p,r}^s)$ (see\,\cite{bcd}):
$$
\widetilde{L}_T^q(\dot{B}_{p,r}^s)={\Big\{}u\in (0,+\infty)\times\mathcal{S}'_h(\R^n):
\|u\|_{\widetilde{L}_T^q(\dot{B}_{p,r}^s)}<+\infty{\Big\}},
$$
where
$$
\|u\|_{\widetilde{L}_T^q(\dot{B}_{p,r}^s)}\stackrel{\mathrm{def}}{=}\bigl{\|}2^{ks}\|\dot{\Delta}_k u(t)\|_{L^q(0,T;L^p)}\bigr{\|}_{\ell^r}.
$$
The index $T$ will be omitted if $T=+\infty$ and we shall denote by $\widetilde{\mathcal{C}}_b([0,T]; \dot{B}^s_{p,r})$ the subset of functions of $\widetilde{L}^\infty_T(\dot{B}^s_{p,r})$ which are also continuous from
$[0,T]$ to $\dot{B}^s_{p,r}$.

By the Minkowski inequality, we have the following inclusions between the
Chemin-Lerner space ${\widetilde{L}^\lambda_{T}(\dot{B}_{p,r}^s)}$ and the Bochner space ${{L}^\lambda_{T}(\dot{B}_{p,r}^s)}$:
\begin{align*}
\|u\|_{\widetilde{L}^\lambda_{T}(\dot{B}_{p,r}^s)}\le\|u\|_{L^\lambda_{T}(\dot{B}_{p,r}^s)}\hspace{0.5cm} \mathrm{if }\hspace{0.2cm}  \lambda\le r,\hspace{0.5cm}
\|u\|_{\widetilde{L}^\lambda_{T}(\dot{B}_{p,r}^s)}\ge\|u\|_{L^\lambda_{T}(\dot{B}_{p,r}^s)},\hspace{0.5cm} \mathrm{if }\hspace{0.2cm}  \lambda\ge r.
\end{align*}
The following Bernstein's lemma will be repeatedly used throughout this paper.

\begin{lemma}\label{bernstein}
Let $\mathcal{B}$ be a ball and $\mathcal{C}$ a ring of $\mathbb{R}^n$. A constant $C$ exists so that for any positive real number $\lambda$, any
non-negative integer k, any smooth homogeneous function $\sigma$ of degree m, and any couple of real numbers $(p, q)$ with
$1\le p \le q\le\infty$, there hold
\begin{align*}
&&\mathrm{Supp} \,\hat{u}\subset\lambda \mathcal{B}\Rightarrow\sup_{|\alpha|=k}\|\partial^{\alpha}u\|_{L^q}\le C^{k+1}\lambda^{k+n(\frac1p-\frac1q)}\|u\|_{L^p},\\
&&\mathrm{Supp} \,\hat{u}\subset\lambda \mathcal{C}\Rightarrow C^{-k-1}\lambda^k\|u\|_{L^p}\le\sup_{|\alpha|=k}\|\partial^{\alpha}u\|_{L^p}
\le C^{k+1}\lambda^{k}\|u\|_{L^p},\\
&&\mathrm{Supp} \,\hat{u}\subset\lambda \mathcal{C}\Rightarrow \|\sigma(D)u\|_{L^q}\le C_{\sigma,m}\lambda^{m+n(\frac1p-\frac1q)}\|u\|_{L^p}.
\end{align*}
\end{lemma}

Next we  recall a few nonlinear estimates in Besov spaces which may be
obtained by means of paradifferential calculus.
Here, we recall the decomposition in the homogeneous context:
\begin{align}\label{bony}
uv=\dot{T}_uv+\dot{T}_vu+\dot{R}(u,v)=\dot{T}_uv+\dot{T}'_vu,
\end{align}
where
$$\dot{T}_uv\stackrel{\mathrm{def}}{=}\sum_{j\in \mathbb{Z}}\dot{S}_{j-1}u\dot{\Delta}_jv, \hspace{0.5cm}\dot{R}(u,v)\stackrel{\mathrm{def}}{=}\sum_{j\in Z}
\dot{\Delta}_ju\widetilde{\dot{\Delta}}_jv,$$
and
$$ \widetilde{\dot{\Delta}}_jv\stackrel{\mathrm{def}}{=}\sum_{|j-j'|\le1}\dot{\Delta}_{j'}v,\hspace{0.5cm} \dot{T}'_vu\stackrel{\mathrm{def}}{=}\sum_{j\in Z}\dot{S}_{j+2}v\dot{\Delta}_ju.$$

The paraproduct $\dot{T}$ and the remainder $\dot{R}$ operators satisfy the following
continuous properties.

\begin{lemma}[\cite{bcd}]\label{fangji}
For all $s\in\mathbb{R}$, $\sigma\ge0$, and $1\leq p, p_1, p_2\leq\infty,$ the
paraproduct $\dot T$ is a bilinear, continuous operator from $\dot{B}_{p_1,1}^{-\sigma}\times \dot{B}_{p_2,1}^s$ to
$\dot{B}_{p,1}^{s-\sigma}$ with $\frac{1}{p}=\frac{1}{p_1}+\frac{1}{p_2}$. The remainder $\dot R$ is bilinear continuous from
$\dot{B}_{p_1, 1}^{s_1}\times \dot{B}_{p_2,1}^{s_2}$ to $
\dot{B}_{p,1}^{s_1+s_2}$ with
$s_1+s_2>0$, and $\frac{1}{p}=\frac{1}{p_1}+\frac{1}{p_2}$.
\end{lemma}

\begin{lemma}\label{daishu}{\rm(\cite[Proposition A.1]{xujiang})}
Let $1\leq p, q\leq \infty$, $s_1\leq \frac{n}{q}$, $s_2\leq n\min\{\frac1p,\frac1q\}$ and $s_1+s_2>n\max\{0,\frac1p +\frac1q -1\}$. For $\forall (u,v)\in\dot{B}_{q,1}^{s_1}({\mathbb R} ^n)\times\dot{B}_{p,1}^{s_2}({\mathbb R} ^n)$, we have
\begin{align*}%\label{product inequality}
\|uv\|_{\dot{B}_{p,1}^{s_1+s_2 -\frac{n}{q}}}\lesssim \|u\|_{\dot{B}_{q,1}^{s_1}}\|v\|_{\dot{B}_{p,1}^{s_2}}.
\end{align*}
\end{lemma}

\begin{lemma}\label{good}
Let   $n\ge 2$ and $
2\leq p \leq \min(4,{2n}/({n-2}))$ and, additionally,$  p\not=4$ { if }\ $n=2$. For any
 $u\in\dot{B}_{p,1}^{\frac {n}{p}}(\R^n), v^\ell\in\dot{B}_{2,1}^{\frac {n}{2}-1 }(\R^n)$ and $ v^h\in\dot{B}_{p,1}^{\frac {n}{p}-1}(\R^n),$ we have
\begin{align}\label{pengyou}
&\|(uv)^\ell\|_{\dot{B}_{2,1}^{\f n2-1}}\lesssim(\|v^\ell\|_{\dot{B}_{2,1}^{\f n2-1}}+\|v^h\|_{\dot{B}_{p,1}^{\frac {n}{p}-1}})\|u\|_{\dot{B}_{p,1}^{\frac {n}{p}}}.
\end{align}
\end{lemma}
\begin{proof}
We first use  Bony's decomposition to write
\begin{align}\label{D9}
\dot S_{j_0+1} (uv)=\dot{T}_{u}\dot S_{j_0+1}  v+\dot S_{j_0+1} \bigl(\dot{T}_{v} u+ \dot{R}(v,u)\bigr)
+[\dot S_{j_0+1} ,\dot{T}_{u}]v.
\end{align}
Applying Lemma \ref{fangji}, we have
\begin{align}\label{D11-1}
\|\dot{T}_{u}\dot S_{j_0+1}  v\|_{\dot{B}_{2,1}^{\f n2-1}}\lesssim&\|u\|_{L^\infty}\|v^\ell\|_{\dot{B}_{2,1}^{\f n2-1}}
\lesssim\|v^\ell\|_{\dot{B}_{2,1}^{\f n2-1}}\|u\|_{\dot{B}_{p,1}^{\frac {n}{p}}},\nonumber
\end{align}
and $(\frac  1{p*}=\frac  12-\frac1p)$
\begin{align}
\|\dot S_{j_0+1} \bigl(\dot{T}_{v} u+ \dot{R}(v,u)\bigr)\|_{\dot{B}_{2,1}^{\f n2-1}}\lesssim&\|v\|_{\dot{B}_{p^*,1}^{\frac {n}{p^*}-1}}\|u\|_{\dot{B}_{p,1}^{\frac {n}{p}}}
\lesssim\|v\|_{\dot{B}_{p,1}^{\frac {n}{p}-1}}\|u\|_{\dot{B}_{p,1}^{\frac {n}{p}}}.
\end{align}
By Lemma 6.1 in \cite{helingbing},  the term about the commutator can be bounded
\begin{align}\label{D11}
\|[\dot S_{j_0+1} ,\dot{T}_{u}]v\|_{\dot{B}_{2,1}^{\f n2-1}}\lesssim&\|\nabla u\|_{\dot{B}_{p^*,1}^{\frac {n}{p^*}-1}}\|v\|_{\dot{B}_{p,1}^{\frac {n}{p}-1}}
\lesssim\|v\|_{\dot{B}_{p,1}^{\frac {n}{p}-1}}\|u\|_{\dot{B}_{p,1}^{\frac {n}{p}}}.
\end{align}
Thus, the combination of (\ref{D9})--(\ref{D11}) shows the validity of  \eqref{pengyou}.

\end{proof}

We also need the following Classical commutator's estimate:
\begin{lemma}{\rm(\cite[Lemma 2.100]{bcd})}\label{jiaohuanzi}
Let $1\leq p\leq \infty$, $-n\min\left\{\frac1p,1-\frac{1}{p}\right\}<s\leq \frac np$. For any
$v\in \dot{B}_{p,1}^{s}(\R^n)$ and $\nabla u\in \dot{B}_{p,1}^{\frac{n}{p}}(\R^n)$,
 there holds
$$
\big\|[\dot{\Delta}_j,u\cdot \nabla ]v\big\|_{L^p}\lesssim d_j 2^{-js}\|\nabla u\|_{\dot{B}_{p,1}^{\frac{n}{p}}}\|v\|_{\dot{B}_{p,1}^{s}}.
$$
\end{lemma}

Finally, we recall a composition result and the parabolic regularity estimate for the
heat equation to end this section.
  \begin{lemma}\label{fuhe}{\rm(\cite{bcd})}
   Let $G$ with $G(0)=0$ be a smooth function defined on an open interval $I$
of $\R$ containing~$0.$
Then  the following estimates
$$
\|G(a)\|_{\dot B^{s}_{p,1}}\lesssim\|a\|_{\dot B^s_{p,1}}\quad\hbox{and}\quad
\|G(a)\|_{\widetilde{L}^q_T(\dot B^{s}_{p,1})}\lesssim\|a\|_{\widetilde{L}^q_T(\dot B^s_{p,1})}
$$
hold true for  $s>0,$ $1\leq p,\, q\leq \infty$ and
 $a$  valued in a bounded interval $J\subset I.$
\end{lemma}
\begin{lemma} [\cite{bcd}]\label{heat}
Let $\sigma\in \R$,  $T>0$, $1\leq p,r\leq\infty$ and $1\leq q_{2}\leq q_{1}\leq\infty$.  Let $u$  satisfy the heat equation
$$\partial_tu-\Delta u=f,\quad
u_{|t=0}=u_0.$$
Then  there holds the following a priori estimate
\begin{align*}
\|u\|_{\widetilde L_{T}^{q_1}(\dot B^{\sigma+\frac 2{q_1}}_{p,r})}\lesssim
\|u_0\|_{\dot B^\sigma_{p,r}}+\|f\|_{\widetilde L^{q_2}_{T}(\dot B^{\sigma-2+\frac 2{q_2}}_{p,r})}.
\end{align*}
\end{lemma}

\section{ The proof of Theorem \ref{dingli1}}
In order to prove the existence part of Theorem \ref{dingli1}, we use a scheme similar to the case of barotropic Navier-Stokes equations, see  \cite{chenqionglei},  \cite{danchin2014},  \cite{haspot} for example. More precisely, let
\begin{align*}
\hbox{$\u_\mathfrak{F}=e^{t\mathcal{A}} \u_0$ \quad with\quad $\mathcal{A} \stackrel{\mathrm{def}}{=}\mu\Delta+(\lambda+\mu)\nabla\div$}, \quad \eta_\mathfrak{F}=e^{\lon t\Delta} \eta_0
\end{align*}
and ${\mathbb{T}}_\mathfrak{F}$ be the solution to the linear system
\begin{align*}
&\partial_t{\mathbb{T}}_\mathfrak{F}+ \frac{A_0}{2\lambda_1}{\mathbb{T}}_\mathfrak{F}-\lon\Delta{\mathbb{T}}_\mathfrak{F}=0,\quad{\mathbb{T}}_\mathfrak{F}(0)={\mathbb{T}}_0 .
\end{align*}
Denote
$
\bar{\u}=\u-\u_\mathfrak{F},\quad \bar{\eta}=\eta-\eta_\mathfrak{F},\quad\bar{{\mathbb{T}}}={\mathbb{T}}-{\mathbb{T}}_\mathfrak{F},
$
then $(a,\bar{\u}, \bar{\eta}, \bar{{\mathbb{T}}} )$ satisfies the following equations:
\begin{eqnarray}\label{linear}
\left\{\begin{aligned}
&\partial_ta+\div \u=-\div(a\u),\\
&\partial_t\bar{\eta}-\varepsilon\Delta\bar{\eta}=-\div(\eta \u),\\
&\partial_t\bar{{\mathbb{T}}}+ \frac{A_0}{2\lambda_1}\bar{{\mathbb{T}}}+(\u\cdot\nabla){\mathbb{T}}-\varepsilon\Delta\bar{{\mathbb{T}}}= \frac{\kappa\,A_0}{2\lambda_1}\eta  \,\mathbb{Id}+F({\mathbb{T}}, \u),\\
&\partial_t\bar{\u}+\u\cdot\nabla \u-\mu\Delta \bar{\u}-(\lambda+\mu)\nabla\div \bar{\u}+\nabla
a=\div{\mathbb{T}}-\kappa L\nabla\eta +G(a,\u,\eta,{\mathbb{T}}).
\end{aligned}\right.
\end{eqnarray}
We use a standard scheme for proving the existence of the solutions.

\begin{itemize}
  \item We smooth out the data and get a sequence of smooth solutions $(a^n,\bar{\u}^n, \bar{\eta}^n, \bar{{\mathbb{T}}}^n )$ of an
approximated system of \eqref{linear}, on a bounded interval $[0; T^n] $ which may depend on
$n$.
  \item We  exhibit a positive lower bound $T$ for $T^n$, and prove uniform estimates on ($a^n$,$\bar{\u}^n$, $\bar{\eta}^n$, $\bar{{\mathbb{T}}}^n $).
  \item We use compactness to prove that the sequence $(a^n,\bar{\u}^n, \bar{\eta}^n, \bar{{\mathbb{T}}}^n )$ converges, up to extraction,
to a solution of \eqref{linear}.
\end{itemize}
Since the proof is
lengthy but standard,  we
refer for instance to \cite{bcd}, \cite{chenqionglei},  \cite{helingbing},
\cite{haspot}, here, we omit the details for brevity. Moreover,  the uniqueness can be obtained in the same way as \cite{danchin2014}.$\hspace{15.8cm}\square$

\section{The proof of Theorem \ref{dingli2}}

This section presents the proof of Theorem \ref{dingli2} stating the global wellposedness for \eqref{m}. The framework is the bootstrap argument, which consists of two main steps.
The first step is to establish the {\it a priori} bounds while the second is to apply
and complete the bootstrap argument by using the {\it a priori} bounds. Main efforts are
devoted to obtaining suitable {\it a priori} bounds.  To do so,  we separate the low frequency
 from the high frequency to distinguish their different behaviors. We make use of some
sharp commutator estimates to shift derivatives.

\subsection{ The estimates in the low frequency}
\noindent The goal of  this subsection is  to establish the {\it a priori} bounds in the low frequency part,
which will be divided  into the following three lemmas.

In the first lemma, we are concerned with the estimates about the polymer number density.
\begin{lemma}\label{jia1}
Under the conditions in Theorem \ref{dingli2}, there holds
\begin{align}\label{er4}
&\|\eta^\ell\|_{\widetilde{L}_t^{\infty}(\dot{B}_{2,1}^{\frac n2-2})}
+ \lon\|\eta^\ell\|_{L^1_t(\dot B^{\frac  n2}_{2,1})}\nonumber\\
&\quad\lesssim\|\eta_0^\ell\|_{\dot{B}_{2,1}^{\frac {n}{2}-2}}
+\int^t_0(\|\eta^\ell\|_{\dot B^{\frac  n2}_{2,1}}+\|\eta^h\|_{\dot B^{\frac  np}_{p,1}})(\|\u^\ell\|_{\dot B^{\frac  n2-1}_{2,1}}+\|\u^h\|_{\dot B^{\frac  np-1}_{p,1}})\,ds.
\end{align}
\end{lemma}
\begin{proof}
Applying $\ddj$ to the second equation of \eqref{m}, then taking the $L^2$ inner product with $\ddj\eta$, we get
\begin{align}\label{er1}
&\frac{1}{2}\frac{d}{dt}\|\dot{\Delta}_j\eta\|_{L^2}^2+c\lon2^{2j}\| \dot{\Delta}_j\eta\|_{L^2}^2
\lesssim \| \dot{\Delta}_j\div(\eta \u)\|_{L^2}\| \dot{\Delta}_j\eta\|_{L^2}
\end{align}
in which we have used the following Bernstein's inequality:
there exists a positive constant $c$ so that
$$
-\int_{\R^n}\Delta\dot{\Delta}_j\eta\cdot\dot{\Delta}_j\eta \,dx\ge c2^{2j}\|\dot{\Delta}_j\eta\|_{L^2}^2.
$$
Multiplying by $1/\| \dot{\Delta}_j\eta\|_{L^2}2^{j(\frac{n}{2}-2)}$  formally on both hand side of
 \eqref{er1},  integrating  the resultant inequality from $0$ to $t$, we can get by summing up about $j\le j_0$ that
\begin{align}\label{er2}
\|\eta^\ell\|_{\widetilde{L}_t^{\infty}(\dot{B}_{2,1}^{\frac n2-2})}
+ \lon\|\eta^\ell\|_{L^1_t(\dot B^{\frac  n2}_{2,1})}
\lesssim&\|\eta_0^\ell\|_{\dot{B}_{2,1}^{\frac {n}{2}-2}}+\int^t_0\|(\div(\eta \u))^\ell\|_{\dot{B}_{2,1}^{\frac n2-2}}\,ds\nonumber\\
\lesssim&\|\eta_0^\ell\|_{\dot{B}_{2,1}^{\frac {n}{2}-2}}+\int^t_0\|(\eta \u)^\ell\|_{\dot{B}_{2,1}^{\frac n2-1}}\,ds.
\end{align}
It follows from Lemma \ref{good} and the embedding relation that
\begin{align}\label{er3}
\|(\eta \u)^\ell\|_{\dot{B}_{2,1}^{\frac n2-1}}
\lesssim&\|\eta\|_{\dot B^{\frac  np}_{p,1}}(\|\u^\ell\|_{\dot B^{\frac  n2-1}_{2,1}}+\|\u^h\|_{\dot B^{\frac  np-1}_{p,1}})\nonumber\\
\lesssim&(\|\eta^\ell\|_{\dot B^{\frac  n2}_{2,1}}+\|\eta^h\|_{\dot B^{\frac  np}_{p,1}})(\|\u^\ell\|_{\dot B^{\frac  n2-1}_{2,1}}+\|\u^h\|_{\dot B^{\frac  np-1}_{p,1}}).
\end{align}
This together with \eqref{er2} gives \eqref{er4}.
\end{proof}

\begin{lemma}\label{jia2}
Under the conditions in Theorem \ref{dingli2}, there holds
\begin{align}\label{er7}
&\|{\mathbb{T}}^\ell\|_{ \widetilde{L}_t^\infty(\dot B^{\frac  n2}_{2,1})}+\frac{A_0}{2\lambda_1} \|{\mathbb{T}}^\ell\|_{L^1_t(\dot B^{\frac  n2}_{2,1})}+ \lon\|{\mathbb{T}}^\ell\|_{L^1_t(\dot B^{\frac  n2+2}_{2,1})}\nonumber\\
&\quad\lesssim \|{\mathbb{T}}_0^\ell\|_{\dot B^{\frac  n2}_{2,1}}+\frac{\kappa\,A_0}{2\lambda_1}\int_0^t\|\eta^\ell\|_{\dot B^{\frac  n2}_{2,1}}\,ds
+\int_0^t(\|{\mathbb{T}}^\ell\|_{\dot B^{\frac  n2}_{2,1}}^2+\|{\mathbb{T}}^h\|_{\dot B^{\frac  np}_{p,1}}^2)\,ds\nonumber\\
&\quad\quad+\int_0^t(\|\u^\ell\|_{\dot B^{\frac  n2-1}_{2,1}}+\|\u^h\|_{\dot B^{\frac  np-1}_{p,1}})(\|\u^\ell\|_{\dot B^{\frac  n2+1}_{2,1}}+\|\u^h\|_{\dot B^{\frac  np+1}_{p,1}})\,ds.
\end{align}
\end{lemma}
\begin{proof}
Along the same derivation of \eqref{er2}, we can deduce from the third equation of \eqref{m} that
\begin{align}\label{er5}
&\|{\mathbb{T}}^\ell\|_{ \widetilde{L}_t^\infty(\dot B^{\frac  n2}_{2,1})}+ \frac{A_0}{2\lambda_1}\|{\mathbb{T}}^\ell\|_{L^1_t(\dot B^{\frac  n2}_{2,1})}+ \lon\|{\mathbb{T}}^\ell\|_{L^1_t(\dot B^{\frac  n2+2}_{2,1})}\nonumber\\
&\quad\lesssim \|{\mathbb{T}}_0^\ell\|_{\dot B^{\frac  n2}_{2,1}}+\frac{\kappa\,A_0}{2\lambda_1}\int_0^t\|\eta^\ell\|_{\dot B^{\frac  n2}_{2,1}}\,ds+\int_0^t\|(F({\mathbb{T}}, \u))^\ell\|_{\dot B^{\frac  n2}_{2,1}}\,ds
+\int_0^t\|(\u\cdot\nabla{\mathbb{T}})^\ell\|_{\dot B^{\frac  n2}_{2,1}}\,ds.
\end{align}
Applying Lemma \ref{good} once again, we have
\begin{align*}%\label{er6}
&\|(F({\mathbb{T}}, \u))^\ell\|_{\dot B^{\frac  n2}_{2,1}}+\|(\u\cdot\nabla{\mathbb{T}})^\ell\|_{\dot B^{\frac  n2-1}_{2,1}}\nonumber\\
&\quad\lesssim\|(F({\mathbb{T}}, \u))^\ell\|_{\dot B^{\frac  n2-1}_{2,1}}+\|(\u\cdot\nabla{\mathbb{T}})^\ell\|_{\dot B^{\frac  n2-1}_{2,1}}\nonumber\\
&\quad\lesssim \|{\mathbb{T}}\|_{\dot B^{\frac  np}_{p,1}}(\|\nabla \u^\ell\|_{\dot B^{\frac  n2-1}_{2,1}}+\|\nabla \u^h\|_{\dot B^{\frac  np-1}_{p,1}})
+ \|\u\|_{\dot B^{\frac  np}_{p,1}}(\|\nabla {\mathbb{T}}^\ell\|_{\dot B^{\frac  n2-1}_{2,1}}+\|\nabla{\mathbb{T}}^h\|_{\dot B^{\frac  np-1}_{p,1}})
\nonumber\\&\quad\lesssim \|{\mathbb{T}}\|_{\dot B^{\frac  np}_{p,1}}(\u^\ell\|_{\dot B^{\frac  n2}_{2,1}}+\| \u^h\|_{\dot B^{\frac  np}_{p,1}})
+ \|\u\|_{\dot B^{\frac  np}_{p,1}}(\| {\mathbb{T}}^\ell\|_{\dot B^{\frac  n2}_{2,1}}+\|{\mathbb{T}}^h\|_{\dot B^{\frac  np}_{p,1}})\nonumber\\
&\quad\lesssim\|{\mathbb{T}}^\ell\|_{\dot B^{\frac  n2}_{2,1}}^2+\|{\mathbb{T}}^h\|_{\dot B^{\frac  np}_{p,1}}^2+(\|\u^\ell\|_{\dot B^{\frac  n2-1}_{2,1}}+\|\u^h\|_{\dot B^{\frac  np-1}_{p,1}})(\|\u^\ell\|_{\dot B^{\frac  n2+1}_{2,1}}+\|\u^h\|_{\dot B^{\frac  np+1}_{p,1}}).
\end{align*}

Inserting the above estimate into \eqref{er5} implies \eqref{er7}.
\end{proof}

\begin{lemma}\label{jia3}
Under the conditions in Theorem \ref{dingli2}, there holds
\begin{align}\label{er14}
&\| (a^\ell, \u^\ell)\|_{ \widetilde{L}_t^\infty(\dot B^{\frac  n2-1}_{2,1})}+ \mu\| (a^\ell, \u^\ell)\|_{L^1_t(\dot B^{\frac  n2+1}_{2,1})}\nonumber\\
&\quad\lesssim\|(a^\ell_0, \u_0^\ell)\|_{\dot B^{\frac  n2-1}_{2,1}}
+ \|{\mathbb{T}}\|^\ell_{L^1_t(\dot B^{\frac  n2}_{2,1})}+\kappa L \|\eta\|^\ell_{L^1_t(\dot B^{\frac  n2}_{2,1})}+\int_0^t \mathcal{N}_1(s)\,ds,
\end{align}
with
\begin{align}\label{er15}
\mathcal{N}_1(t)\stackrel{\mathrm{def}}{=}&\big(\| (a^\ell,\u^\ell)\|_{\dot B^{\frac  n2-1}_{2,1}}+\| \u^h\|_{\dot B^{\frac  np-1}_{p,1}}\big)\big(\| (a^\ell,\u^\ell)\|_{\dot B^{\frac  n2+1}_{2,1}}+\| \u^h\|_{\dot B^{\frac  np+1}_{p,1}}+\|\eta^\ell\|_{\dot B^{\frac  n2}_{2,1}}+\|\eta^h\|_{\dot B^{\frac  np}_{p,1}}\big)\nonumber\\
&+\Big(\|a^\ell\|_{\dot B^{\frac  n2-1}_{2,1}}+\|{\mathbb{T}}^\ell\|_{\dot B^{\frac  n2}_{2,1}}+\|(a^h,{\mathbb{T}}^h)\|_{\dot B^{\frac  np}_{p,1}}\Big)\Big(\|(\eta^\ell,{\mathbb{T}}^\ell)\|_{\dot B^{\frac  n2}_{2,1}}+\| (a^h,{\mathbb{T}}^h)\|_{\dot B^{\frac  np}_{p,1}}\nonumber\\
&\quad\quad\quad\quad\quad\quad\quad\quad+\|\eta^h\|_{\dot B^{\frac  np+1}_{p,1}}+\|\eta^\ell\|_{\dot B^{\frac  n2-2}_{2,1}}\|\eta^\ell\|_{\dot B^{\frac  n2}_{2,1}}+\|\eta^h\|_{\dot B^{\frac  np-1}_{p,1}}\|\eta^h\|_{\dot B^{\frac  np+1}_{p,1}}\Big).
\end{align}
\end{lemma}
\begin{proof}
The combination of the first equation and the fourth equation in \eqref{m} is similar to the compressible Navier-Stokes equation up to some nonlinear terms, thus, by using the operators $\p$ and $\q$, we get the equation of incompressible part
\begin{align}\label{tian1}
\partial_t\p \u -\mu\Delta \p \u=-\p (\u\cdot\nabla \u)+\p \div{\mathbb{T}}+\p G(a,\u,\eta,{\mathbb{T}}),
\end{align}
and the equation of compressible part
\begin{eqnarray}\label{tian2}
\left\{\begin{aligned}
&\partial_ta+\div \u=-\div(a\u),\\
&\partial_t\q \u -\nu\Delta \q \u+\nabla
a=-\q(\u\cdot\nabla \u)+\q\div{\mathbb{T}}-\kappa L\nabla\eta +\q G(a,\u,\eta,{\mathbb{T}}).
\end{aligned}\right.
\end{eqnarray}
Applying $\dot{\Delta}_{j}$ to   \eqref{tian1} and taking  $L^2$ inner product of the resulting equation with $\dot{\Delta}_{j}\p {\u}$,
applying the H\"older inequality and
integrating the resultant inequality over $[0, t]$, then multiplying the inequality by $2^{(\frac{n}{2}-1)j}$ and taking summation for $j\le j_0$, we arrive at
\begin{align}\label{er8}
\|\p \u^\ell\|_{ \widetilde{L}_t^\infty(\dot B^{\frac  n2-1}_{2,1})}+&\mu \|\p \u^\ell\|_{L^1_t(\dot B^{\frac  n2+1}_{2,1})}
\lesssim\|\p \u^\ell_0\|_{\dot B^{\frac  n2-1}_{2,1}}
+ \|(\p \div{\mathbb{T}})^\ell\|_{L^1_t(\dot B^{\frac  n2-1}_{2,1})}\nonumber\\
&+ \|(\p (\u\cdot\nabla  \u))^\ell\|_{L^1_t(\dot B^{\frac  n2-1}_{2,1})}+ \|(\p G(a,\u,\eta,{\mathbb{T}}))^\ell\|_{L^1_t(\dot B^{\frac  n2-1}_{2,1})}.
\end{align}

For the equations in \eqref{tian2}, an energy estimates for the barotropic linearized equations (see \cite{bcd}, Prop. 10.23, or \cite{danchin2000}) thus give
\begin{align*}
&\|(a,\q \u)^\ell\|_{\widetilde{L}_t^{\infty}(\dot{B}_{2,1}^{\frac n2-1})}
+\nu\|(a,\q \u)^\ell\|_{L^1_t(\dot{B}_{2,1}^{\frac n2+1})}\nonumber\\
&\quad\lesssim\|( a_0,\q \u_0)^\ell\|_{\dot{B}_{2,1}^{\frac {n}{2}-1}}+\int^t_0\|(\div{\mathbb{T}})^\ell\|_{\dot{B}_{2,1}^{\frac n2-1}}\,ds+\kappa L\int^t_0\|(\nabla\eta)^\ell\|_{\dot{B}_{2,1}^{\frac n2-1}}\,ds\nonumber\\
&\quad\quad+\int^t_0\|(\div(a\u))^\ell\|_{\dot{B}_{2,1}^{\frac n2-1}}\,ds+\int^t_0\|(\u\cdot\nabla \u)^\ell\|_{\dot{B}_{2,1}^{\frac n2-1}}\,ds+\int^t_0\|( G(a,\u,\eta,{\mathbb{T}}))^\ell\|_{\dot{B}_{2,1}^{\frac n2-1}}\,ds.
\end{align*}
Summing up the above two estimates implies that
\begin{align}\label{er9}
&\| (a^\ell, \u^\ell)\|_{ \widetilde{L}_t^\infty(\dot B^{\frac  n2-1}_{2,1})}+\mu \| (a^\ell, \u^\ell)\|_{L^1_t(\dot B^{\frac  n2+1}_{2,1})}\nonumber\\
&\quad\lesssim\|(a^\ell_0, \u_0^\ell)\|_{\dot B^{\frac  n2-1}_{2,1}}
+ \|{\mathbb{T}}\|^\ell_{L^1_t(\dot B^{\frac  n2}_{2,1})}+\kappa L \|\eta\|^\ell_{L^1_t(\dot B^{\frac  n2}_{2,1})}\nonumber\\
&\quad\quad+ \| (\u\cdot\nabla  \u)\|^\ell_{L^1_t(\dot B^{\frac  n2-1}_{2,1})}+ \| \div(a\u)\|^\ell_{L^1_t(\dot B^{\frac  n2-1}_{2,1})}+ \| G(a,\u,\eta,{\mathbb{T}})\|^\ell_{L^1_t(\dot B^{\frac  n2-1}_{2,1})}.
\end{align}
Next, we deal with nonlinear terms on the right hand side of \eqref{er9}.

At first, thanks to Lemma \ref{good}, one has
\begin{align*}
\|(\u\cdot\nabla  \u)^\ell\|_{\dot B^{\frac  n2-1}_{2,1}}\lesssim&(\| \u^\ell\|_{\dot B^{\frac  n2-1}_{2,1}}+\| \u^h\|_{\dot B^{\frac  np-1}_{p,1}})\|\nabla \u\|_{\dot B^{\frac  np}_{p,1}}\nonumber\\
\lesssim&(\| \u^\ell\|_{\dot B^{\frac  n2-1}_{2,1}}+\| \u^h\|_{\dot B^{\frac  np-1}_{p,1}})(\| \u^\ell\|_{\dot B^{\frac  n2+1}_{2,1}}+\| \u^h\|_{\dot B^{\frac  np+1}_{p,1}}).
\end{align*}
Similarly, from Lemmas \ref{good} and \ref{fuhe}, there hold the following three estimates:
\begin{align*}%\label{er11}
\|(\div(a\u))^\ell\|_{\dot B^{\frac  n2-1}_{2,1}}\lesssim&\| \u\|_{\dot B^{\frac  np}_{p,1}}(\|  a^h\|_{\dot B^{\frac  np}_{p,1}}+\|  a^\ell\|_{\dot B^{\frac  n2}_{2,1}})+\| a\|_{\dot B^{\frac  np}_{p,1}}(\|  \u^h\|_{\dot B^{\frac  np}_{p,1}}+\|  \u^\ell\|_{\dot B^{\frac  n2}_{2,1}})\nonumber\\
\lesssim&\| \u^\ell\|_{\dot B^{\frac  n2}_{2,1}}^2+\| \u^h\|_{\dot B^{\frac  np}_{p,1}}^2+\| a^\ell\|_{\dot B^{\frac  n2}_{2,1}}^2+\| a^h\|_{\dot B^{\frac  np}_{p,1}}^2\nonumber\\
\lesssim&(\| (a^\ell,\u^\ell)\|_{\dot B^{\frac  n2-1}_{2,1}}\!\!+\| \u^h\|_{\dot B^{\frac  np-1}_{p,1}})(\| (a^\ell,\u^\ell)\|_{\dot B^{\frac  n2+1}_{2,1}}+\| \u^h\|_{\dot B^{\frac  np+1}_{p,1}})\!\!+\| a^h\|_{\dot B^{\frac  np}_{p,1}}^2,
\end{align*}
\begin{align*}%\label{er12}
\| I(a)( \div{\mathbb{T}}- \nabla\eta)\|^\ell_{\dot{B}_{2,1}^{\frac {n}{2}-1}}
\lesssim&\| I(a)\|_{\dot{B}_{p,1}^{\frac {n}{p}}}(\| (\div{\mathbb{T}}- \nabla\eta)^\ell\|_{\dot{B}_{2,1}^{\frac {n}{2}-1}}+\| (\div{\mathbb{T}}- \nabla\eta)^h\|_{\dot{B}_{p,1}^{\frac {n}{p}-1}})\nonumber\\
\lesssim&(\|a^\ell\|_{\dot B^{\frac  n2-1}_{2,1}}\!\!+\|a^h\|_{\dot B^{\frac  np}_{p,1}})(\|{\mathbb{T}}^\ell\|_{\dot B^{\frac  n2}_{2,1}}\!\!+\|{\mathbb{T}}^h\|_{\dot B^{\frac  np}_{p,1}}\!\!+\|\eta^\ell\|_{\dot B^{\frac  n2}_{2,1}}\!\!+\|\eta^h\|_{\dot B^{\frac  np+1}_{p,1}}),
\end{align*}
\begin{align*}%\label{er13}
\| (1-I(a))\eta\nabla\eta\|^\ell_{\dot{B}_{2,1}^{\frac {n}{2}-1}}
\lesssim&(1+\| I(a)\|_{\dot{B}_{p,1}^{\frac {n}{p}}})(\| (\eta\nabla\eta)^\ell\|_{\dot{B}_{2,1}^{\frac {n}{2}-1}}+\| (\eta\nabla\eta)^h\|_{\dot{B}_{p,1}^{\frac {n}{p}-1}})\nonumber\\
\lesssim&(\|a^\ell\|_{\dot B^{\frac  n2-1}_{2,1}}+\|a^h\|_{\dot B^{\frac  np}_{p,1}})(\|\eta^\ell\|_{\dot B^{\frac  n2-2}_{2,1}}\|\eta^\ell\|_{\dot B^{\frac  n2}_{2,1}}+\|\eta^h\|_{\dot B^{\frac  np-1}_{p,1}}\|\eta^h\|_{\dot B^{\frac  np+1}_{p,1}}).
\end{align*}
Along the same lines, one can deal with the rest two terms in $G(a,\u,\eta,{\mathbb{T}})$, thus, substituting the above inequality into \eqref{er9} leads to  the lemma, and we complete
the proof of Lemma \ref{jia3}.
\end{proof}

\subsection{ The estimates in the high frequency}
In this subsection, we are concerned with the estimates in the high frequency  part.

\begin{lemma}\label{jia4}
Under the conditions in Theorem \ref{dingli2},   there holds
\begin{align}\label{er33}
&\|( \u^h,\eta^h)\|_{ \widetilde{L}_t^\infty(\dot B^{\frac  np-1}_{p,1})}+\|(a^h,{\mathbb{T}}^h)\|_{ \widetilde{L}_t^\infty(\dot B^{\frac  np}_{p,1})}+  \mu\|\u^h\|_{L^1_t(\dot B^{\frac  np+1}_{p,1})}+ \lon\|\eta^h\|_{L^1_t(\dot B^{\frac  np+1}_{p,1})}\nonumber\\
&\quad\quad+ \lon\|{\mathbb{T}}^h\|_{L^1_t(\dot B^{\frac  np+2}_{p,1})}+ \nu^{-1}\|a^h\|_{L^1_t(\dot B^{\frac  np}_{p,1})}+ \frac{A_0}{2\lambda_1}\|{\mathbb{T}}^h\|_{L^1_t(\dot B^{\frac  np}_{p,1})}\nonumber\\
&\quad\lesssim \|(\u_0^h,\eta_0^h)\|_{\dot B^{\frac  np-1}_{p,1}}+\|(a_0^h,{\mathbb{T}}_0^h)\|_{\dot B^{\frac  np}_{p,1}}+\int_0^t \mathcal{N}_1(s)\,ds
\end{align}
with $\mathcal{N}_1(t)$ defined in \eqref{er15}.
\end{lemma}
\begin{proof}

By a standard energy estimates, we can get from \eqref{tian1} that
\begin{align}\label{er16}
&\|\p \u^h\|_{ \widetilde{L}_t^\infty(\dot B^{\frac  np-1}_{p,1})}+\mu \|\p \u^h\|_{L^1_t(\dot B^{\frac  np+1}_{p,1})}\nonumber\\
&\quad\lesssim\|\p \u^h_0\|_{\dot B^{\frac  np-1}_{p,1}}
+ \| (\div{\mathbb{T}})^h\|_{L^1_t(\dot B^{\frac  np-1}_{p,1})}\nonumber\\
&\quad\quad+ \|(\u\cdot\nabla  \u)^h\|_{L^1_t(\dot B^{\frac  np-1}_{p,1})}+ \| (G(a,\u,\eta,{\mathbb{T}}))^h\|_{L^1_t(\dot B^{\frac  np-1}_{p,1})}.
\end{align}
By the product law in Lemma \ref{daishu}, we have
\begin{align}\label{er18}
\|(\u\cdot\nabla  \u)^h\|_{\dot B^{\frac  np-1}_{p,1}}\lesssim&\|\u\|_{\dot B^{\frac  np-1}_{p,1}}\|\u\|_{\dot B^{\frac  np+1}_{p,1}}\nonumber\\
\lesssim&(\| \u^\ell\|_{\dot B^{\frac  n2-1}_{2,1}}+\| \u^h\|_{\dot B^{\frac  np-1}_{p,1}})(\| \u^\ell\|_{\dot B^{\frac  n2+1}_{2,1}}+\| \u^h\|_{\dot B^{\frac  np+1}_{p,1}}).
\end{align}
The  first two terms of $G(a,\u,\eta,{\mathbb{T}})$ can be estimated in the same way as \eqref{er18}. To bound the third term of $G(a,\u,\eta,{\mathbb{T}})$, we use Lemmas \ref{daishu}, \ref{fuhe} and interpolation inequality to get
\begin{align}\label{er19}
&\| I(a)( \div{\mathbb{T}}- \nabla\eta)\|^h_{\dot{B}_{p,1}^{\frac {n}{p}-1}}\lesssim\| I(a)\|_{\dot{B}_{p,1}^{\frac {n}{p}}}\|  \div{\mathbb{T}}- \nabla\eta\|_{\dot{B}_{p,1}^{\frac {n}{p}-1}}\nonumber\\
&\quad\lesssim\| I(a)\|_{\dot{B}_{p,1}^{\frac {n}{p}}}(\| (\div{\mathbb{T}}- \nabla\eta)^\ell\|_{\dot{B}_{2,1}^{\frac {n}{2}-1}}+\| (\div{\mathbb{T}}- \nabla\eta)^h\|_{\dot{B}_{p,1}^{\frac {n}{p}-1}})\nonumber\\
&\quad\lesssim(\|a^\ell\|_{\dot B^{\frac  n2-1}_{2,1}}+\|a^h\|_{\dot B^{\frac  np}_{p,1}})(\|{\mathbb{T}}^\ell\|_{\dot B^{\frac  n2}_{2,1}}+\|{\mathbb{T}}^h\|_{\dot B^{\frac  np}_{p,1}}+\|\eta^\ell\|_{\dot B^{\frac  n2}_{2,1}}+\|\eta^h\|_{\dot B^{\frac  np}_{p,1}})
\nonumber\\
&\quad\lesssim(\|a^\ell\|_{\dot B^{\frac  n2-1}_{2,1}}+\|a^h\|_{\dot B^{\frac  np}_{p,1}})(\|{\mathbb{T}}^\ell\|_{\dot B^{\frac  n2}_{2,1}}+\|{\mathbb{T}}^h\|_{\dot B^{\frac  np}_{p,1}}+\|\eta^\ell\|_{\dot B^{\frac  n2}_{2,1}}+\|\eta^h\|_{\dot B^{\frac  np+1}_{p,1}}).
\end{align}
Similarly,
\begin{align}\label{er20}
&\|(1-I(a))\eta\nabla\eta\|^h_{\dot{B}_{p,1}^{\frac {n}{p}-1}}\nonumber\\
&\quad\lesssim(1+\| I(a)\|_{\dot{B}_{p,1}^{\frac {n}{p}}})(\| (\eta\nabla\eta)^\ell\|_{\dot{B}_{2,1}^{\frac {n}{2}-1}}+\| (\eta\nabla\eta)^h\|_{\dot{B}_{p,1}^{\frac {n}{p}-1}})\nonumber\\
&\quad\lesssim(\|a^\ell\|_{\dot B^{\frac  n2-1}_{2,1}}+\|a^h\|_{\dot B^{\frac  np}_{p,1}})(\|\eta^\ell\|_{\dot B^{\frac  n2}_{2,1}}^2+\|\eta^h\|_{\dot B^{\frac  np}_{p,1}}^2)\nonumber\\
&\quad\lesssim(\|a^\ell\|_{\dot B^{\frac  n2-1}_{2,1}}+\|a^h\|_{\dot B^{\frac  np}_{p,1}})(\|\eta^\ell\|_{\dot B^{\frac  n2-2}_{2,1}}\|\eta^\ell\|_{\dot B^{\frac  n2}_{2,1}}+\|\eta^h\|_{\dot B^{\frac  np-1}_{p,1}}\|\eta^h\|_{\dot B^{\frac  np+1}_{p,1}}).
\end{align}
Inserting the above estimates into \eqref{er16}, we have
\begin{align}\label{er21}
\|\p \u^h\|_{ \widetilde{L}_t^\infty(\dot B^{\frac  np-1}_{p,1})}+\mu \|\p \u^h\|_{L^1_t(\dot B^{\frac  np+1}_{p,1})}
\lesssim&\|\p \u^h_0\|_{\dot B^{\frac  np-1}_{p,1}}
+ \| {\mathbb{T}}\|^h_{L^1_t(\dot B^{\frac  np}_{p,1})}+\int_0^t \mathcal{N}_1(s)\,ds
\end{align}
with $\mathcal{N}_1(t)$ defined in \eqref{er15}.

Similarly, from the second equation of \eqref{m}, we have
\begin{align}\label{er22}
&\|\eta^h\|_{ \widetilde{L}_t^\infty(\dot B^{\frac  np-1}_{p,1})}+\lon \|\eta^h\|_{L^1_t(\dot B^{\frac  np+1}_{p,1})}\nonumber\\
&\quad\lesssim \|\eta_0^h\|_{\dot B^{\frac  np-1}_{p,1}}+\int_0^t\|(\eta\div \u)^h\|_{\dot B^{\frac  np-1}_{p,1}}\,ds
\nonumber\\
&\quad\quad+\int_0^t\|\div \u\|_{L^\infty}\|\eta^h\|_{\dot B^{\frac  np-1}_{p,1}}\,ds+\int_0^t\sum_{j\ge j_0}2^{(\frac {n}{p}-1)j}\|[\ddj,\u\cdot\nabla]\eta\|_{L^p}\,ds,
\end{align}
and from the third equation of \eqref{m}, we have
\begin{align}\label{er23}
&\|{\mathbb{T}}^h\|_{ \widetilde{L}_t^\infty(\dot B^{\frac  np}_{p,1})}+\frac{A_0}{2\lambda_1}
 \|{\mathbb{T}}^h\|_{L^1_t(\dot B^{\frac  np}_{p,1})}+ \lon\|{\mathbb{T}}^h\|_{L^1_t(\dot B^{\frac  np+2}_{p,1})}\nonumber\\
&\quad\lesssim \|{\mathbb{T}}_0^h\|_{\dot B^{\frac  np}_{p,1}}+\frac{\kappa\,A_0}{2\lambda_1}\int_0^t\|\eta^h\|_{\dot B^{\frac  np}_{p,1}}\,ds+\int_0^t\|(F({\mathbb{T}}, \u))^h\|_{\dot B^{\frac  np}_{p,1}}\,ds
\nonumber\\
&\quad\quad+\int_0^t\|\div \u\|_{L^\infty}\|{\mathbb{T}}^h\|_{\dot B^{\frac  np}_{p,1}}\,ds+\int_0^t\sum_{j\ge j_0}2^{\frac {n}{p}j}\|[\ddj,\u\cdot\nabla]{\mathbb{T}}\|_{L^p}\,ds.
\end{align}
Due to Lemma \ref{daishu} and Lemma \ref{jiaohuanzi}, we get
\begin{align}\label{er24}
&\sum_{j\ge j_0}2^{(\frac {n}{p}-1)j}\|[\ddj,\u\cdot\nabla]\eta\|_{L^p}+\|(\eta\div \u)^h\|_{\dot B^{\frac  np-1}_{p,1}}\nonumber\\
&\quad\lesssim\|\eta\|_{\dot B^{\frac  np-1}_{p,1}}\|\nabla \u\|_{\dot B^{\frac  np}_{p,1}}\nonumber\\
&\quad\lesssim(\|\eta^\ell\|_{\dot B^{\frac  n2-2}_{2,1}}+\|\eta^h\|_{\dot B^{\frac  np-1}_{p,1}})(\|\u^\ell\|_{\dot B^{\frac  n2+1}_{2,1}}+\|\u^h\|_{\dot B^{\frac  np+1}_{p,1}})
\end{align}
and
\begin{align}\label{er25}
&\sum_{j\ge j_0}2^{\frac {n}{p}j}\|[\ddj,\u\cdot\nabla]{\mathbb{T}}\|_{L^p}+\|(F({\mathbb{T}}, \u))^h\|_{\dot B^{\frac  np}_{p,1}}\nonumber\\
&\quad\lesssim\|{\mathbb{T}}\|_{\dot B^{\frac  np}_{p,1}}\|\nabla \u\|_{\dot B^{\frac  np}_{p,1}}\nonumber\\
&\quad\lesssim(\|{\mathbb{T}}^\ell\|_{\dot B^{\frac  n2}_{2,1}}+\|{\mathbb{T}}^h\|_{\dot B^{\frac  np}_{p,1}})(\|\u^\ell\|_{\dot B^{\frac  n2+1}_{2,1}}+\|\u^h\|_{\dot B^{\frac  np+1}_{p,1}}),
\end{align}
from which, we get by summing up \eqref{er22} and \eqref{er23} that
\begin{align}\label{er26}
&\|\eta^h\|_{ \widetilde{L}_t^\infty(\dot B^{\frac  np-1}_{p,1})}+\|{\mathbb{T}}^h\|_{ \widetilde{L}_t^\infty(\dot B^{\frac  np}_{p,1})}+ \frac{A_0}{2\lambda_1}\|{\mathbb{T}}^h\|_{L^1_t(\dot B^{\frac  np}_{p,1})}+\frac{\lon}{2} \|\eta^h\|_{L^1_t(\dot B^{\frac  np+1}_{p,1})}+\lon\|{\mathbb{T}}^h\|_{L^1_t(\dot B^{\frac  np+2}_{p,1})}\nonumber\\
&\quad\lesssim \|\eta_0^h\|_{\dot B^{\frac  np-1}_{p,1}}+\|{\mathbb{T}}_0^h\|_{\dot B^{\frac  np}_{p,1}}
\nonumber\\
&\quad\quad+\int_0^t(\|\eta^\ell\|_{\dot B^{\frac  n2-2}_{2,1}}+\|{\mathbb{T}}^\ell\|_{\dot B^{\frac  n2}_{2,1}}+\|\eta^h\|_{\dot B^{\frac  np-1}_{p,1}}+\|{\mathbb{T}}^h\|_{\dot B^{\frac  np}_{p,1}})(\|\u^\ell\|_{\dot B^{\frac  n2+1}_{2,1}}+\|\u^h\|_{\dot B^{\frac  np+1}_{p,1}})\,ds.
\end{align}

To estimate the high frequency part of $(a,\q \u)$, we follow the method used in  \cite{chenzhimin}, \cite{haspot} to introduce a new quantity
$$\ga\stackrel{\mathrm{def}}{=}\q \u+\nu^{-1}(-\Delta)^{-1}\nabla a,$$
 from which and the second equation of \eqref{tian2}, we have
\begin{align}\label{er27}
\partial_t \ga-\nu\Delta \ga
=&\nu^{-1}\ga-\nu^{-2}(-\Delta)^{-1}\nabla a\nonumber\\
&+\nu^{-1}\q(a\u)-\q(\u\cdot\nabla \u)+\q\div{\mathbb{T}}-\kappa L\nabla\eta +\q G(a,\u,\eta,{\mathbb{T}}).
\end{align}

We   get by a standard energy argument that
\begin{align}\label{er28}
&\|\ga^h\|_{ \widetilde{L}_t^\infty(\dot B^{\frac  np-1}_{p,1})}+ \nu\|\ga^h\|_{L^1_t(\dot B^{\frac  np+1}_{p,1})}\nonumber\\
&\quad\lesssim \|\ga_0^h\|_{\dot B^{\frac  np-1}_{p,1}}+\int_0^t(\nu^{-1}\|\ga^h\|_{\dot B^{\frac  np-1}_{p,1}}+\nu^{-2}\|a^h\|_{\dot B^{\frac  np-2}_{p,1}})\,ds+\int_0^t\| {\mathbb{T}}^h\|_{\dot B^{\frac  np}_{p,1}}\,ds+\kappa L\int_0^t\|\eta^h\|_{\dot B^{\frac  np}_{p,1}}\,ds\nonumber\\
&\quad\quad+\int_0^t\|(\u\cdot\nabla \u)^h\|_{\dot B^{\frac  np-1}_{p,1}}\,ds+\int_0^t\|(G(a,\u,\eta,{\mathbb{T}}))^h\|_{\dot B^{\frac  np-1}_{p,1}}\,ds+\nu^{-1}\int_0^t\|( a\u)^h\|_{\dot B^{\frac  np-1}_{p,1}}\,ds.
\end{align}
Plugging $\q \u=\ga-\nu^{-1}(-\Delta)^{-1}\nabla a$ into the first equation in \eqref{tian2} gives
\begin{align}\label{er29}
&\partial_t a +\nu^{-1}a+\u\cdot \nabla a=-\div \ga-a\div \u,
\end{align}
which further implies that
\begin{align}\label{er30}
&\|a^h\|_{ \widetilde{L}_t^\infty(\dot B^{\frac  np}_{p,1})}+ \nu^{-1}\|a^h\|_{L^1_t(\dot B^{\frac  np}_{p,1})}\nonumber\\
&\quad\lesssim \|a_0^h\|_{\dot B^{\frac  np}_{p,1}}+\int_0^t\|\ga^h\|_{\dot B^{\frac  np+1}_{p,1}}\,ds+\int_0^t\|(a\div \u)^h\|_{\dot B^{\frac  np+1}_{p,1}}\,ds
\nonumber\\
&\quad\quad+\int_0^t\|\div \u\|_{L^\infty}\|a^h\|_{\dot B^{\frac  np}_{p,1}}\,ds+\int_0^t\sum_{j\ge j_0}2^{\frac {n}{p}j}\|[\ddj,\u\cdot\nabla]a\|_{L^p}\,ds.
\end{align}
From Lemma \ref{daishu} and interpolation inequality, we have
\begin{align*}%\label{er31}
\|( a\u)^h\|_{\dot B^{\frac  np-1}_{p,1}}\lesssim&\| ( a\u)^h\|_{\dot B^{\frac  np}_{p,1}}\lesssim\| a\|_{\dot B^{\frac  np}_{p,1}}\| \u\|_{\dot B^{\frac  np}_{p,1}}\nonumber\\
\lesssim&\| a\|_{\dot B^{\frac  np}_{p,1}}^2+\| \u\|_{\dot B^{\frac  np}_{p,1}}^2\nonumber\\
\lesssim&\| a^\ell\|_{\dot B^{\frac  n2-1}_{2,1}}\| a^\ell\|_{\dot B^{\frac  n2+1}_{2,1}}+\| a^h\|_{\dot B^{\frac  np}_{p,1}}^2+\| \u^\ell\|_{\dot B^{\frac  n2-1}_{2,1}}\| \u^\ell\|_{\dot B^{\frac  n2+1}_{2,1}}+\| \u^h\|_{\dot B^{\frac  np-1}_{p,1}}\| \u^h\|_{\dot B^{\frac  np+1}_{p,1}}.
\end{align*}
The other nonlinear terms on the right hand side of \eqref{er28}, \eqref{er30} can be estimates  similarly to \eqref{er18}, \eqref{er19}, \eqref{er24}, \eqref{er25},
thus, the combination of \eqref{er28} and \eqref{er30} implies that
\begin{align}\label{er32}
&\|a^h\|_{ \widetilde{L}_t^\infty(\dot B^{\frac  np}_{p,1})}+\|\ga^h\|_{ \widetilde{L}_t^\infty(\dot B^{\frac  np-1}_{p,1})} +\frac{1}{2\nu}\|a^h\|_{L^1_t(\dot B^{\frac  np}_{p,1})}+ \frac\nu2\|\ga^h\|_{L^1_t(\dot B^{\frac  np+1}_{p,1})}\nonumber\\
&\quad\lesssim \|a_0^h\|_{\dot B^{\frac  np}_{p,1}}+\|\ga_0^h\|_{\dot B^{\frac  np-1}_{p,1}}+\|{\mathbb{T}}^h\|_{L^1_t(\dot B^{\frac  np}_{p,1})}+\kappa L\|\eta^h\|_{L^1_t(\dot B^{\frac  np}_{p,1})}+\int_0^t \mathcal{N}_1(s)\,ds.
\end{align}
%with $\mathcal{N}_1(t)$ defined in \eqref{er15}.

Combining with \eqref{er21}, \eqref{er26} and \eqref{er32} and using
$\q \u=\ga-\nu^{-1}(-\Delta)^{-1}\nabla a$, we can obtain \eqref{er33}.

\end{proof}

\subsection{ Complete the proof of  Theorem \ref{dingli2}}
Now, we can complete the proof of our main Theorem \ref{dingli2} by the continuous arguments.
To accelerate the proof, we first
denote
\begin{align*}%\label{yi5}
X(t)\stackrel{\mathrm{def}}{=}&\|(a,\u)\|^\ell_{\widetilde{L}^\infty_t(\dot{B}_{2,1}^{\frac{n}{2}-1})}
+\|\eta\|^\ell_{\widetilde{L}^\infty_t(\dot{B}_{2,1}^{\frac{n}{2}-2})}
+\|{\mathbb{T}}\|^\ell_{\widetilde{L}^\infty_t(\dot{B}_{2,1}^{\frac{n}{2}})}
+\|(\u,\eta)\|^h_{\widetilde{L}^\infty_t(\dot{B}_{p,1}^{\frac{n}{p}-1})}
\nonumber\\
&\quad+\|(a,{\mathbb{T}})\|^h_{\widetilde{L}^\infty_t(\dot{B}_{p,1}^{\frac{n}{p}})}+\|(a,\u)\|^\ell_{L^1_t(\dot{B}_{2,1}^{\frac{n}{2}+1})}+\|(\eta,{\mathbb{T}})\|^\ell_{L^1_t(\dot{B}_{2,1}^{\frac{n}{2}})}+ \|{\mathbb{T}}^\ell\|_{L^1_t(\dot B^{\frac  n2+2}_{2,1})}\nonumber\\
&\quad\quad+\|(a,{\mathbb{T}})\|^h_{L^1_t(\dot{B}_{p,1}^{\frac{n}{p}})}+\|(\u,\eta)\|^h_{L^1_t(\dot{B}_{p,1}^{\frac{n}{p}+1})}
+\|{\mathbb{T}}\|^h_{L^1_t(\dot{B}_{p,1}^{\frac{n}{p}+2})},
\nonumber\\
X_0\stackrel{\mathrm{def}}{=}&
\|(a^\ell_0, \u_0^\ell)\|_{\dot B^{\frac  n2-1}_{2,1}}+\|\eta_0^\ell\|_{\dot{B}_{2,1}^{\frac {n}{2}-2}}+\|{\mathbb{T}}_0^\ell\|_{\dot B^{\frac  n2}_{2,1}}+\|(\u_0^h,\eta_0^h)\|_{\dot B^{\frac  np-1}_{p,1}}+\|(a_0^h,{\mathbb{T}}_0^h)\|_{\dot B^{\frac  np}_{p,1}}.
\end{align*}

Multiplying by a suitable large constant on both sides of \eqref{er7} and then pulsing \eqref{er4}, we can  finally get by combining  the resulting inequality with
\eqref{er14} that
\begin{align}\label{er14+1}
&\| (a^\ell, \u^\ell)\|_{ \widetilde{L}_t^\infty(\dot B^{\frac  n2-1}_{2,1})}+\|\eta^\ell\|_{\widetilde{L}_t^{\infty}(\dot{B}_{2,1}^{\frac n2-2})}+\|{\mathbb{T}}^\ell\|_{ \widetilde{L}_t^\infty(\dot B^{\frac  n2}_{2,1})}
\nonumber\\
&\quad\quad
+ \mu\| (a^\ell, \u^\ell)\|_{L^1_t(\dot B^{\frac  n2+1}_{2,1})}+\frac{\lon}{2}\|\eta^\ell\|_{L^1_t(\dot B^{\frac  n2}_{2,1})}+ \frac{\lon}{2}\|{\mathbb{T}}^\ell\|_{L^1_t(\dot B^{\frac  n2+2}_{2,1})}+\frac{A_0}{4\lambda_1}
 \|{\mathbb{T}}^\ell\|_{L^1_t(\dot B^{\frac  n2}_{2,1})}\nonumber\\
&\quad\lesssim\|(a^\ell_0, \u_0^\ell)\|_{\dot B^{\frac  n2-1}_{2,1}}+\|\eta_0^\ell\|_{\dot{B}_{2,1}^{\frac {n}{2}-2}}+\|{\mathbb{T}}_0^\ell\|_{\dot B^{\frac  n2}_{2,1}}
+\int_0^t \mathcal{N}_1(s)\,ds.
\end{align}

Next,
combining with \eqref{er33} and \eqref{er14+1}, we can get
\begin{align}\label{energy3}
X(t)\le X_0+C(X(t))^2(1+CX(t)).
\end{align}

Under the setting of initial data in Theorem\ref{dingli2},  there exists a positive constant $C_0$ such that
$X_0\leq C_0 \epsilon$. Due to the local existence result which has been  achieved by Theorem \ref{dingli1}, there exists a positive time $T$ such that
\begin{equation}\label{re}
 X (t) \leq 2 C_0\ \epsilon , \quad  \forall \; t \in [0, T].
\end{equation}
Let $T^{*}$ be the largest possible time of $T$ for what \eqref{re} holds. Now, we only need to show $T^{*} = \infty$.  By the estimate of \eqref{energy3}, we can use
 a standard continuation argument to prove that $T^{*} = \infty$ provided that $\epsilon$ is small enough.  We omit the details here. Hence, we finish the proof of Theorem \ref{dingli2}. $\hspace{12.5cm}\square$

\section{The proof of Theorem \ref{dingli3}}
In this section, we shall follow the method (independent of the spectral analysis) used in \cite{guoyan} and \cite{xujiang2019arxiv}  to get the decay rate of the solutions constructed in the previous section. For convenience, we assume all the coefficients appeared in \eqref{sys} equal to one.
From the proof of Theorem \ref{dingli2}, we can get the following inequality (see the derivation of \eqref{er14+1} and \eqref{er33} for more details):
 \begin{align}\label{sa1}
&\frac{d}{dt}(\|(a,\u)\|^\ell_{\dot{B}_{2,1}^{\frac{n}{2}-1}}
+\|\eta\|^\ell_{\dot{B}_{2,1}^{\frac{n}{2}-2}}
+\|{\mathbb{T}}\|^\ell_{\dot{B}_{2,1}^{\frac{n}{2}}}
+\|(\u,\eta)\|^h_{\dot{B}_{p,1}^{\frac{n}{p}-1}}
+\|(a,{\mathbb{T}})\|^h_{\dot{B}_{p,1}^{\frac{n}{p}}})
\nonumber\\
&\quad+\|(a,\u)\|^\ell_{\dot{B}_{2,1}^{\frac{n}{2}+1}}+\|(\eta,{\mathbb{T}})\|^\ell_{\dot{B}_{2,1}^{\frac{n}{2}}}
+\|a\|^h_{\dot{B}_{p,1}^{\frac{n}{p}}}+\|(\u,\eta)\|^h_{\dot{B}_{p,1}^{\frac{n}{p}+1}}
+\|{\mathbb{T}}\|^h_{\dot{B}_{p,1}^{\frac{n}{p}+2}}\nonumber\\
&\lesssim(\| (a^\ell,\u^\ell)\|_{\dot B^{\frac  n2-1}_{2,1}}+\| \u^h\|_{\dot B^{\frac  np-1}_{p,1}})(\| \u^h\|_{\dot B^{\frac  np+1}_{p,1}}+\| (a^\ell,\u^\ell)\|_{\dot B^{\frac  n2+1}_{2,1}}+\|\eta^\ell\|_{\dot B^{\frac  n2}_{2,1}}+\|\eta^h\|_{\dot B^{\frac  np}_{p,1}})\nonumber\\
&\quad+(\|a^\ell\|_{\dot B^{\frac  n2-1}_{2,1}}+\|(a^h,{\mathbb{T}}^h)\|_{\dot B^{\frac  np}_{p,1}})(\| (a^h,{\mathbb{T}}^h)\|_{\dot B^{\frac  np}_{p,1}}+\|\eta^\ell\|_{\dot B^{\frac  n2-2}_{2,1}}\|\eta^\ell\|_{\dot B^{\frac  n2}_{2,1}}+\|\eta^h\|_{\dot B^{\frac  np-1}_{p,1}}\|\eta^h\|_{\dot B^{\frac  np+1}_{p,1}})\nonumber\\
&\quad+(\|a^\ell\|_{\dot B^{\frac  n2-1}_{2,1}}+\|{\mathbb{T}}^\ell\|_{\dot B^{\frac  n2}_{2,1}}+\| a^h\|_{\dot B^{\frac  np}_{p,1}})(\|{\mathbb{T}}^\ell\|_{\dot B^{\frac  n2}_{2,1}}+\| {\mathbb{T}}^h\|_{\dot B^{\frac  np}_{p,1}}+\|\eta^\ell\|_{\dot B^{\frac  n2}_{2,1}}+\|\eta^h\|_{\dot B^{\frac  np+1}_{p,1}}).
\end{align}
By Theorem \ref{dingli2}, the following estimate holds:
\begin{align}\label{sa2}
&\|(a,\u)\|^\ell_{\widetilde{L}^\infty_t(\dot{B}_{2,1}^{\frac{n}{2}-1})}
+\|\eta\|^\ell_{\widetilde{L}^\infty_t(\dot{B}_{2,1}^{\frac{n}{2}-2})}
+\|{\mathbb{T}}\|^\ell_{\widetilde{L}^\infty_t(\dot{B}_{2,1}^{\frac{n}{2}})}
+\|(\u,\eta)\|^h_{\widetilde{L}^\infty_t(\dot{B}_{p,1}^{\frac{n}{p}-1})}
+\|(a,{\mathbb{T}})\|^h_{\widetilde{L}^\infty_t(\dot{B}_{p,1}^{\frac{n}{p}})}\leq c_0,
\end{align}
from which we can infer from  \eqref{sa1}  that
\begin{align}\label{sa3}
&\frac{d}{dt}(\|(a,\u)\|^\ell_{\dot{B}_{2,1}^{\frac{n}{2}-1}}
+\|\eta\|^\ell_{\dot{B}_{2,1}^{\frac{n}{2}-2}}
+\|{\mathbb{T}}\|^\ell_{\dot{B}_{2,1}^{\frac{n}{2}}}
+\|(\u,\eta)\|^h_{\dot{B}_{p,1}^{\frac{n}{p}-1}}
+\|(a,{\mathbb{T}})\|^h_{\dot{B}_{p,1}^{\frac{n}{p}}})
\nonumber\\
&\quad+\bar{c}(\|(a,\u)\|^\ell_{\dot{B}_{2,1}^{\frac{n}{2}+1}}+\|(\eta,{\mathbb{T}})\|^\ell_{\dot{B}_{2,1}^{\frac{n}{2}}}
+\|a\|^h_{\dot{B}_{p,1}^{\frac{n}{p}}}+\|(\u,\eta)\|^h_{\dot{B}_{p,1}^{\frac{n}{p}+1}}
+\|{\mathbb{T}}\|^h_{\dot{B}_{p,1}^{\frac{n}{p}+2}})\le0.
\end{align}
In order to derive the decay estimate of the solutions given in Theorem \ref{dingli2}, we need  to get  a Lyapunov-type differential inequality from \eqref{sa3}.

 According to \eqref{sa2} and  embedding relation in the high frequency,
it's obvious  for any $\beta>0$ that
\begin{align}\label{sa52}
\|{\mathbb{T}}\|^h_{\dot{B}_{2,1}^{\frac n2}}\ge C \big(\|{\mathbb{T}}\|^h_{\dot{B}_{2,1}^{\frac n2}}\big)^{1+\beta},
\end{align}
and
\begin{align}\label{sa4}
\|a\|^h_{\dot{B}_{p,1}^{\frac{n}{p}}}+\|(\u,\eta)\|^h_{\dot{B}_{p,1}^{\frac{n}{p}+1}}
+\|{\mathbb{T}}\|^h_{\dot{B}_{p,1}^{\frac{n}{p}+2}}\ge C(\|(\u,\eta)\|^h_{\dot{B}_{p,1}^{\frac{n}{p}-1}}
+\|(a,{\mathbb{T}})\|^h_{\dot{B}_{p,1}^{\frac{n}{p}}})^{1+\beta}.
\end{align}
Thus, to get the Lyapunov-type inequality of the solutions, we only need  to control the norm of $\|(a,\u)\|^\ell_{\dot{B}_{2,1}^{\frac{n}{2}+1}}+\|(\eta,{\mathbb{T}})\|^\ell_{\dot{B}_{2,1}^{\frac{n}{2}}}$.
This process can be obtained from  the fact that   the  solutions constructed in Theorem \ref{dingli2} can
propagate the regularity of the
initial data in Besov space with low regularity, see the following Proposition \ref{propagate}. This will ensure that one can use interpolation to get the desired Lyapunov-type inequality.
\begin{proposition}\label{propagate}
Let $(a,\u,\eta,{\mathbb{T}}) $ be the  solutions constructed in Theorem \ref{dingli2}. For any $\frac{n}{2}-\frac{2n}{p}\le\sigma<\frac{n}{2}-1,$ and $(a_0^{\ell},\u_0^{\ell})\in {\dot{B}^{\sigma}_{2,\infty}}(\R^n), \eta_0^{\ell}\in {\dot{B}^{\sigma-1}_{2,\infty}}(\R^n), {\mathbb{T}}_0^{\ell}\in{\dot{B}^{\sigma+1}_{2,\infty}}(\R^n),$
then there exists a constant $C_0>0$ depends on the norm of the initial data
such that for all $t\geq0$,
 \begin{eqnarray}\label{sa39-1}
\|(a,\u)(t,\cdot)\|^{\ell}_{\dot B^{\sigma}_{2,\infty}}+\|\eta(t,\cdot)\|^{\ell}_{\dot B^{\sigma-1}_{2,\infty}}+\|{\mathbb{T}}(t,\cdot)\|^{\ell}_{\dot B^{\sigma+1}_{2,\infty}}\leq C_0.
\end{eqnarray}
\end{proposition}
\begin{proof}
To  simplify the process of proof,  we  first define the nonlinear terms in \eqref{m} as
 \begin{align*}
 f_1\stackrel{\mathrm{def}}{=}&-\div(a\u),\quad f_2\stackrel{\mathrm{def}}{=}-\div(\eta \u),\nonumber\\
 f_3\stackrel{\mathrm{def}}{=}&-(\u\cdot\nabla){\mathbb{T}}+F({\mathbb{T}}, \u),\quad f_4\stackrel{\mathrm{def}}{=}-(\u\cdot\nabla) \u +G(a,\eta,\u,{\mathbb{T}}).
\end{align*}

From the first   and fourth equations of \eqref{m}, we get by a similar derivation of Lemma 5.1 in \cite{xujiang2019arxiv} that
\begin{align}\label{sa9}
\|(a,\u)\|^{\ell}_{\dot{B}^{\sigma}_{2,\infty}}
\lesssim\|(a_0,\u_0)\|^{\ell}_{\dot{B}^{\sigma}_{2,\infty}}
+\|(\eta,{\mathbb{T}})\|^{\ell}_{\widetilde{L}^1_t(\dot{B}^{\sigma+1}_{2,\infty})}
+\int_0^t\|(f_1,f_4)\|^{\ell}_{\dot{B}^{\sigma}_{2,\infty}}\,ds.
\end{align}
From the second and  third equations of \eqref{m}, we get by using Lemma \ref{heat} in the low frequency that
\begin{align}\label{sa11}
\|\eta\|^{\ell}_{\dot{B}^{\sigma-1}_{2,\infty}}+\|\eta\|^{\ell}_{\widetilde{L}^1_t(\dot{B}^{\sigma+1}_{2,\infty})}
\lesssim\|\eta_0\|^{\ell}_{\dot{B}^{\sigma-1}_{2,\infty}}+\int_0^t\|f_2\|^{\ell}_{\dot{B}^{\sigma-1}_{2,\infty}}\,ds,
\end{align}
and
\begin{align}\label{qi14}
\|{\mathbb{T}}\|^{\ell}_{\dot{B}^{\sigma+1}_{2,\infty}}+\|{\mathbb{T}}\|^{\ell}_{\widetilde{L}^1_t(\dot{B}^{\sigma+1}_{2,\infty})}
\lesssim&\|{\mathbb{T}}_0\|^{\ell}_{\dot{B}^{\sigma+1}_{2,\infty}}+\|\eta\|^{\ell}_{\widetilde{L}^1_t(\dot{B}^{\sigma+1}_{2,\infty})}
+\int_0^t\|f_3\|^{\ell}_{\dot{B}^{\sigma+1}_{2,\infty}}\,ds.
\end{align}
Multiplying by a suitable large constant on both sides of \eqref{sa11} and then pulsing \eqref{qi14}, we can  finally get by combining  the resulting inequality with \eqref{sa9} that
\begin{align}\label{sa11+11}
&\|(a,\u) \|^{\ell}_{\dot{B}^{\sigma}_{2,\infty}}+\|\eta\|^{\ell}_{\dot{B}^{\sigma-1}_{2,\infty}}
+\|{\mathbb{T}}\|^{\ell}_{\dot{B}^{\sigma+1}_{2,\infty}}\nonumber\\
&\quad\lesssim\|(a_0,\u_0)\|^{\ell}_{\dot{B}^{\sigma}_{2,\infty}}+\|\eta_0\|^{\ell}_{\dot{B}^{\sigma-1}_{2,\infty}}+\|{\mathbb{T}}_0\|^{\ell}_{\dot{B}^{\sigma+1}_{2,\infty}}
\nonumber\\&\quad\quad+\int_0^t\|(f_1,f_4)\|^{\ell}_{\dot{B}^{\sigma}_{2,\infty}}\,ds+\int_0^t\|f_2\|^{\ell}_{\dot{B}^{\sigma-1}_{2,\infty}}\,ds
+\int_0^t\|f_3\|^{\ell}_{\dot{B}^{\sigma+1}_{2,\infty}}\,ds.
\end{align}

To estimate the nonlinear terms in $f_1, f_2, f_3, f_4$, we claim the following important
estimates, which we shall postpone its proof in the Appendix.

\noindent { $\mathrm{\mathbf{Claim:}}$}
Let
 $n= 2, 3$ {and} $ 2\leq p \leq \min(4,{2n}/({n-2}))$ additionally, $  p\not=4\ \hbox{ if }\ n=2$, we have
\begin{align}
&\|fg\|_{\dot{B}_{2,\infty}^{\sigma}}\!\!\lesssim \|f\|_{\dot{B}_{2,\infty}^{\sigma}}\|g\|_{\dot{B}_{p,1}^{\frac np}},\quad -\frac{n}{p}\le \sigma<\frac{n}{p}.\label{key0}\\
&\|f g^\ell\|_{\dot B^{\sigma}_{2,\infty}}^\ell\!\! \lesssim\|f\|_{\dot B^{\frac np-1}_{p,1}}\|g^\ell\|_{\dot B^{\sigma+1}_{2,\infty}},\quad \frac{n}{2}-\frac{2n}{p}\le\sigma<\frac{n}{2}-1.
\label{key1}\\
&\|f g^h\|_{\dot B^{\sigma}_{2,\infty}}^\ell \!\!\lesssim\|f\|_{\dot B^{\frac np-1}_{p,1}}\|g^h\|_{\dot B^{\frac np}_{p,1}},\quad \frac{n}{2}-\frac{2n}{p}\le\sigma<\frac{n}{2}-1.\label{key}
\end{align}
To accelerate the proof, we continue to introduce the following notation:
\begin{align*}%\label{sa7}
&\mathcal{E}_\infty(t)\stackrel{\mathrm{def}}{=}\|(a,\u)\|^\ell_{\dot{B}_{2,1}^{\frac{n}{2}-1}}
+\|\eta\|^\ell_{\dot{B}_{2,1}^{\frac{n}{2}-2}}
+\|{\mathbb{T}}\|^\ell_{\dot{B}_{2,1}^{\frac{n}{2}}}
+\|(\u,\eta)\|^h_{\dot{B}_{p,1}^{\frac{n}{p}-1}}
+\|(a,{\mathbb{T}})\|^h_{\dot{B}_{p,1}^{\frac{n}{p}}},
\nonumber\\
&\mathcal{E}_1(t)\stackrel{\mathrm{def}}{=}\|(a,\u)\|^\ell_{\dot{B}_{2,1}^{\frac{n}{2}+1}}+\|(\eta,{\mathbb{T}})\|^\ell_{\dot{B}_{2,1}^{\frac{n}{2}}}
+\|a\|^h_{\dot{B}_{p,1}^{\frac{n}{p}}}+\|(\u,\eta)\|^h_{\dot{B}_{p,1}^{\frac{n}{p}+1}}
+\|{\mathbb{T}}\|^h_{\dot{B}_{p,1}^{\frac{n}{p}+2}}.
\end{align*}
From \eqref{key0}, one has
\begin{align}\label{sa15}
&\|\u\!\cdot\!\nabla a^{\ell}\|^{\ell}_{\dot{B}^{\sigma}_{2,\infty}}\!+\!\|a\,\mathrm{div}\,\u^{\ell}\|^{\ell}_{\dot{B}^{\sigma}_{2,\infty}}
\nonumber\\&\quad
%\lesssim\|u^\ell\!\cdot\!\nabla a^{\ell}\|^{\ell}_{\dot{B}^{\sigma}_{2,\infty}}\!+\!\|\u^h\!\cdot\!\nabla %a^{\ell}\|^{\ell}_{\dot{B}^{\sigma}_{2,\infty}}\!+\!\|a^{\ell}\,\mathrm{div}\,\u^{\ell}\|^{\ell}_{\dot{B}^{\sigma}_{2,\infty}}
%\!+\!\|a^{h}\,\mathrm{div}\,\u^{\ell}\|^{\ell}_{\dot{B}^{\sigma}_{2,\infty}}\nonumber\\
%&\quad
\lesssim\|\u^\ell\|_{\dot{B}^\sigma_{2,\infty}} \|\nabla a^\ell\|_{\dot{B}^\frac{n}{p}_{p,1}}
\!+\!\|\u\|^h_{\dot B^\frac{n}{p}_{p,1}}\|\nabla a^{\ell}\|_{\dot B^\sigma_{2,\infty}}
\!+\!\|a^\ell\|_{\dot B^\sigma_{2,\infty}}\|\mathrm{div}\,\u^\ell\|_{\dot B^\frac{n}{p}_{p,1}}
\!+\!\|a^h\|_{\dot B^\frac{n}{p}_{p,1}}\|\mathrm{div}\,\u^{\ell}\|_{\dot B^\sigma_{2,\infty}}\nonumber\\
&\quad\lesssim\|a^{\ell}\|_{\dot{B}^{\frac n2+1}_{2,1}}\|\u^{\ell}\|_{\dot{B}^{\sigma}_{2,\infty}}\!+\!\|\u\|^{h}_{\dot{B}^{\frac {n}{p}\!+\!1}_{p,1}}\|a^{\ell}\|_{\dot{B}^{\sigma}_{2,\infty}}\!+\!\|\u\|^{\ell}_{\dot{B}^{\frac n2+1}_{2,1}}\|a\|^{\ell}_{\dot{B}^{\sigma}_{2,\infty}}\!+\!\|a\|^{h}_{\dot{B}^{\frac {n}{p}}_{p,1}}\|\u\|^{\ell}_{\dot{B}^{\sigma}_{2,\infty}}\nonumber\\
&\quad\lesssim\er\|(a^{\ell},\u^{\ell})\|_{\dot{B}^{\sigma}_{2,\infty}}.
\end{align}
Using the second estimate in  \eqref{key1}, we have
\begin{align}\label{sa16}
\|\u\cdot&\nabla a^h\|_{\dot B^{\sigma}_{2,\infty}}^\ell\!+\!\|a\mathrm{div}\u^{h}\|_{\dot B^{\sigma}_{2,\infty}}^\ell
\nonumber\\
\lesssim& \|\u\|^{\ell}_{\dot B^{\sigma+1}_{2,\infty}}\|\nabla a\|^h_{\dot B^{\frac {n}{p}-1}_{p,1}}+\!\|\nabla a\|^h_{\dot B^{\frac {n}{p}-1}_{p,1}}\|\u\|^h_{\dot B^{\frac {n}{p}}_{p,1}}
 \!+\!(\|a\|^\ell_{\dot B^{\frac {n}{p}-1}_{p,1}}\!+\!\|a\|^h_{\dot B^{\frac {n}{p}-1}_{p,1}})\|\mathrm{div}\u\|^h_{\dot B^{\frac {n}{p}}_{p,1}}\nonumber
\\
%\lesssim& (\|\u\|^{\ell}_{\dot B^{\frac n2-1}_{2,1}}\!+\!\|\u\|^h_{\dot B^{\frac {n}{p}-1}_{p,1}})\|a\|^h_{\dot B^{\frac {n}{p}}_{p,1}}
 %\!+\! (\|a\|^{\ell}_{\dot B^{\frac n2-1}_{2,1}}\!+\!\|a\|^h_{\dot B^{\frac {n}{p}-1}_{p,1}})\|\u\|^h_{\dot B^{\frac {n}{p}}_{p,1}}\nonumber\\
\lesssim& \|a\|^h_{\dot B^{\frac {n}{p}}_{p,1}}\|\u\|^{\ell}_{\dot B^{\sigma}_{2,\infty}}\!+\!(\|a\|^{\ell}_{\dot B^{\frac n2-1}_{2,1}}+\|a\|^h_{\dot B^{\frac {n}{p}}_{p,1}})\|\u\|^h_{\dot B^{\frac {n}{p}+1}_{p,1}}\lesssim\er\|\u^{\ell}\|_{\dot{B}^{\sigma}_{2,\infty}}+\e\er
\end{align}
from which  and \eqref{sa15}
% with the identity
%$
%\div (a \u)=\u\cdot\nabla a^\ell+a\mathrm{div}\u^{\ell}+\u\cdot\nabla a^h+a\mathrm{div}\u^{h},
%$
gives
\begin{align}\label{sa18}
\|f_1\|^{\ell}_{\dot{B}^{\sigma}_{2,\infty}}
\lesssim\er\|(a^{\ell},\u^{\ell})\|_{\dot{B}^{\sigma}_{2,\infty}}+\e\er.
\end{align}
Along the same lines, we have
\begin{align}\label{sa12}
\|f_2\|^{\ell}_{\dot{B}^{\sigma-1}_{2,\infty}}&\lesssim\|\div(\eta \u)\|^{\ell}_{\dot{B}^{\sigma-1}_{2,\infty}}
%\lesssim\|\eta \u\|^{\ell}_{\dot{B}^{\sigma}_{2,\infty}}\nonumber\\
%&
\lesssim\|\eta^{\ell} \u^{\ell}\|^{\ell}_{\dot{B}^{\sigma}_{2,\infty}}+\|\eta^{h} \u^{\ell}\|^{\ell}_{\dot{B}^{\sigma}_{2,\infty}}+\|\eta^{\ell} \u^h\|^{\ell}_{\dot{B}^{\sigma}_{2,\infty}}+\|\eta^{h} \u^{h}\|^{\ell}_{\dot{B}^{\sigma}_{2,\infty}}\nonumber\\
&\lesssim\| \u^{\ell}\|_{\dot{B}^{\sigma}_{2,\infty}}(\|\eta^{\ell}\|_{\dot B^{\frac n2}_{2,1}}+\|\eta^{h}\|_{\dot B^{\frac np+1}_{p,1}})+\| \eta^{\ell}\|_{\dot{B}^{\sigma-1}_{2,\infty}}\|\u^{h}\|_{\dot B^{\frac np+1}_{p,1}}+\|\eta^{h}\|_{\dot B^{\frac np-1}_{p,1}}\|\u^{h}\|_{\dot B^{\frac np+1}_{p,1}}\nonumber\\
&\lesssim
\er(\|\u \|^{\ell}_{\dot{B}^{\sigma}_{2,\infty}}+\|\eta\|^{\ell}_{\dot{B}^{\sigma-1}_{2,\infty}})+\e\er.
\end{align}
To bound the terms in $f_3$, we use the decomposition $\u=\u^\ell+\u^h$, {and}$ {\mathbb{T}}={\mathbb{T}}^\ell+{\mathbb{T}}^h$ to get
\begin{align}\label{sa19}
&\|\u\cdot\nabla {\mathbb{T}}^{\ell}\|^{\ell}_{\dot{B}^{\sigma+1}_{2,\infty}}+\|{\mathbb{T}}\,\mathrm{div}\,\u^{\ell}\|^{\ell}_{\dot{B}^{\sigma+1}_{2,\infty}}
+\|\u^{\ell}\cdot\nabla {\mathbb{T}}^h\|_{\dot B^{\sigma+1}_{2,\infty}}^\ell+\|{\mathbb{T}}^{\ell}\mathrm{div}\u^{h}\|_{\dot B^{\sigma+1}_{2,\infty}}^\ell\nonumber\\&\quad
\lesssim\|\u^\ell\|_{\dot{B}^\sigma_{2,\infty}} \|\nabla {\mathbb{T}}^\ell\|_{\dot{B}^\frac{n}{p}_{p,1}}
+\|\u\|^h_{\dot B^\frac{n}{p}_{p,1}}\|\nabla {\mathbb{T}}^{\ell}\|_{\dot B^{\sigma}_{2,\infty}}
+\|{\mathbb{T}}\|_{\dot B^\frac{n}{p}_{p,1}}\|\mathrm{div}\,\u^{\ell}\|_{\dot B^\sigma_{2,\infty}}\nonumber\\
&\quad\quad+\|\u^\ell\|_{\dot{B}^\sigma_{2,\infty}} \|\nabla {\mathbb{T}}^h\|_{\dot{B}^\frac{n}{p}_{p,1}}+\|{\mathbb{T}}^\ell\|_{\dot{B}^{\sigma+1}_{2,\infty}} \| \div \u^h\|_{\dot{B}^{\frac{n}{p}-1}_{p,1}}\nonumber\\
&\quad\lesssim(\|{\mathbb{T}}^{\ell}\|_{\dot{B}^{\frac n2}_{2,1}}+\|{\mathbb{T}}\|^{h}_{\dot{B}^{\frac {n}{p}+2}_{p,1}})\|\u^{\ell}\|_{\dot{B}^{\sigma}_{2,\infty}}+(\|\u\|^{\ell}_{\dot{B}^{\frac n2+1}_{2,1}}+\|\u\|^{h}_{\dot{B}^{\frac {n}{p}+1}_{p,1}})\|{\mathbb{T}}\|^{\ell}_{\dot{B}^{\sigma+1}_{2,\infty}}\nonumber\\
&\quad\lesssim\er(\|\u^{\ell}\|_{\dot{B}^{\sigma}_{2,\infty}}+\|{\mathbb{T}}\|^{\ell}_{\dot{B}^{\sigma+1}_{2,\infty}}).
\end{align}
It follows from \eqref{key} that
\begin{align}\label{sa20}
\|\u^h\cdot\nabla {\mathbb{T}}^h\|_{\dot B^{\sigma+1}_{2,\infty}}^\ell+\|{\mathbb{T}}^h\mathrm{div}\u^{h}\|_{\dot B^{\sigma+1}_{2,\infty}}^\ell
\lesssim&\|\u^h\cdot\nabla {\mathbb{T}}^h\|_{\dot B^{\sigma}_{2,\infty}}^\ell+\|{\mathbb{T}}^h\mathrm{div}\u^{h}\|_{\dot B^{\sigma}_{2,\infty}}^\ell
\nonumber\\
\lesssim& \|\u\|^h_{\dot B^{\frac {n}{p}-1}_{p,1}}\|\nabla {\mathbb{T}}\|^h_{\dot B^{\frac {n}{p}}_{p,1}}
 +\|{\mathbb{T}}\|^h_{\dot B^{\frac {n}{p}-1}_{p,1}}\|\mathrm{div}\u\|^h_{\dot B^{\frac {n}{p}}_{p,1}}\nonumber
\\
\lesssim& \|\u\|^h_{\dot B^{\frac {n}{p}-1}_{p,1}}\|{\mathbb{T}}\|^h_{\dot B^{\frac {n}{p}+2}_{p,1}}+\|{\mathbb{T}}\|^h_{\dot B^{\frac {n}{p}}_{p,1}}\|\u\|^h_{\dot B^{\frac {n}{p}+1}_{p,1}}\nonumber\\
\lesssim&\e\er.
\end{align}
The term of $\|F({\mathbb{T}}, \u)\|^{\ell}_{\dot{B}^{\sigma+1}_{2,\infty}}$ can be bounded the same as \eqref{sa19}, \eqref{sa20}, thus we can get
\begin{align}\label{sa21}
\|f_3\|^{\ell}_{\dot{B}^{\sigma+1}_{2,\infty}}
\lesssim\er(\|\u^{\ell}\|_{\dot{B}^{\sigma}_{2,\infty}}+\|{\mathbb{T}}\|^{\ell}_{\dot{B}^{\sigma+1}_{2,\infty}})+\e\er.
\end{align}
Finally, we have to bound the terms in $f_4$. Indeed, the terms $\u\cdot\nabla \u,$ $k(a)\nabla a$, $I(a)(\Delta \u+\nabla\div \u)$, $\eta\nabla\eta$ in $f_4$ can be dealt with the same as \eqref{sa15}, \eqref{sa16}, \eqref{sa12}, \eqref{sa19}, here we omit the details. We only present  some  representative terms in the following.
According to the definition of $I(a)$, it's not hard to check that $$I(a)=a-aI(a).$$
Now, using \eqref{key1}, \eqref{key} and Lemma \ref{fuhe}, we get
\begin{align}\label{sa23}
\|aI(a)\div{\mathbb{T}} \|^{\ell}_{\dot{B}^{\sigma}_{2,\infty}}
\lesssim&\|a^\ell I(a)\div{\mathbb{T}} \|^{\ell}_{\dot{B}^{\sigma}_{2,\infty}}+\|a^h I(a)\div{\mathbb{T}} \|^{\ell}_{\dot{B}^{\sigma}_{2,\infty}}\nonumber\\
\lesssim&\|I(a)\div{\mathbb{T}}\|_{\dot{B}^{\frac np}_{p,1}} \|a \|^{\ell}_{\dot{B}^{\sigma}_{2,\infty}}+\|a^h\|_{\dot{B}^{\frac np}_{p,1}}\|I(a)\div{\mathbb{T}} \|_{\dot{B}^{\frac np-1}_{p,1}}\nonumber\\
\lesssim&(\|a^\ell\|_{\dot{B}^{\frac n2-1}_{2,1}} +\|a^h\|_{\dot{B}^{\frac np}_{p,1}} )(\|{\mathbb{T}}^\ell\|_{\dot{B}^{\frac n2}_{2,1}} +\|{\mathbb{T}}^h\|_{\dot{B}^{\frac np+2}_{p,1}} )\|a \|^{\ell}_{\dot{B}^{\sigma}_{2,\infty}}\nonumber\\
&+(\|a^\ell\|_{\dot{B}^{\frac n2-1}_{2,1}} +\|a^h\|_{\dot{B}^{\frac np}_{p,1}} )(\|{\mathbb{T}}^\ell\|_{\dot{B}^{\frac n2}_{2,1}} +\|{\mathbb{T}}^h\|_{\dot{B}^{\frac np}_{p,1}} )\|a^{h} \|_{\dot{B}^{\frac np}_{p,1}},
\end{align}
from which
\begin{align}\label{sa24}
\|I(a)\div{\mathbb{T}} \|^{\ell}_{\dot{B}^{\sigma}_{2,\infty}}
\lesssim&(1+\e)\er\|a \|^{\ell}_{\dot{B}^{\sigma}_{2,\infty}}+(1+\e)\e\er.
\end{align}
Similarly,
\begin{align}\label{sa26}
\|I(a)\nabla \eta\|^{\ell}_{\dot{B}^{\sigma}_{2,\infty}}
\lesssim&(1+\e)\er\|a \|^{\ell}_{\dot{B}^{\sigma}_{2,\infty}}+(1+\e)\e\er.
\end{align}
Thanks to our claim \eqref{key0}--\eqref{key} again, we have
\begin{align}\label{sa28}
\|I(a)\eta\nabla \eta \|^{\ell}_{\dot{B}^{\sigma}_{2,\infty}}
\lesssim&\|I(a)\|_{\dot{B}^{\frac np}_{p,1}}\|\eta\nabla \eta \|^{\ell}_{\dot{B}^{\sigma}_{2,\infty}}
\nonumber\\
\lesssim&(\|a^\ell\|_{\dot{B}^{\frac n2-1}_{2,1}} +\|a^h\|_{\dot{B}^{\frac np}_{p,1}} )(\| \eta^\ell\|_{\dot{B}^{\frac n2}_{2,1}} \|\eta\|^{\ell}_{\dot{B}^{\sigma-1}_{2,\infty}}+\|\eta^h\|_{\dot{B}^{\frac np-1}_{p,1}}\| \eta^h \|_{\dot{B}^{\frac np+1}_{p,1}})\nonumber\\
\lesssim&\e\er \|\eta\|^{\ell}_{\dot{B}^{\sigma-1}_{2,\infty}}+(\e)^2\er.
\end{align}
In all, collecting the above estimates, we obtain that
\begin{align}\label{sa33}
\|f_4\|^{\ell}_{\dot{B}^{\sigma}_{2,\infty}}
\lesssim&(1+\e)\er(\|(a,\u) \|^{\ell}_{\dot{B}^{\sigma}_{2,\infty}}+\|\eta\|^{\ell}_{\dot{B}^{\sigma-1}_{2,\infty}})+(1+\e)\e\er.
\end{align}

Inserting \eqref{sa18}, \eqref{sa20}, \eqref{sa21}, \eqref{sa33} into \eqref{sa11+11} gives
\begin{align}\label{sa36}
&\|(a,\u) \|^{\ell}_{\dot{B}^{\sigma}_{2,\infty}}+\|\eta\|^{\ell}_{\dot{B}^{\sigma-1}_{2,\infty}}
+\|{\mathbb{T}}\|^{\ell}_{\dot{B}^{\sigma+1}_{2,\infty}}\nonumber\\
&\quad\lesssim\|(a_0,\u_0)\|^{\ell}_{\dot{B}^{\sigma}_{2,\infty}}+\|\eta_0\|^{\ell}_{\dot{B}^{\sigma-1}_{2,\infty}}+\|{\mathbb{T}}_0\|^{\ell}_{\dot{B}^{\sigma+1}_{2,\infty}}
+\int_0^t(1+\mathcal{E}_\infty(s))\mathcal{E}_\infty(s)\mathcal{E}_1(s) \,ds\nonumber\\
&\quad\quad+\int_0^t
(1+\mathcal{E}_\infty(s))\mathcal{E}_1(s)(\|(a,\u) \|^{\ell}_{\dot{B}^{\sigma}_{2,\infty}}+\|\eta\|^{\ell}_{\dot{B}^{\sigma-1}_{2,\infty}}+\|{\mathbb{T}}\|^{\ell}_{\dot{B}^{\sigma+1}_{2,\infty}})\,ds.
\end{align}
Thus, one can employ nonlinear generalisations of the Gronwall's inequality to get
\begin{eqnarray}\label{sa39}
\|(a,\u)(t,\cdot)\|^{\ell}_{\dot B^{\sigma}_{2,\infty}}+\|\eta(t,\cdot)\|^{\ell}_{\dot B^{\sigma-1}_{2,\infty}}+\|{\mathbb{T}}(t,\cdot)\|^{\ell}_{\dot B^{\sigma+1}_{2,\infty}}\leq C_0
\end{eqnarray}
for all $t\geq0$, where $C_0>0$ depends on the norm of the initial data.

Consequently, we complete the proof of Proposition \ref{propagate}.
\end{proof}
Now, we are begin to
get   Lyapunov-type  inequality from \eqref{sa3}.
For any $\frac{n}{2}-\frac{2n}{p}\le\sigma<\frac{n}{2}-1,$
it follows from interpolation inequality that
\begin{align*}%\label{sa5}
\|(a,\u)\|^\ell_{\dot{B}_{2,1}^{\frac n2-1}}
\le& C \big(\|(a,\u)\|^\ell_{\dot{B}_{2,\infty}^{\sigma}}\big)^{\theta_{1}}\big(\|(a,\u)\|^\ell_{\dot{B}_{2,1}^{\frac n2+1}}\big)^{1-\theta_{1}}\nonumber\\
\le&C\big(\|(a,\u)\|_{\dot{B}_{2,\infty}^{\sigma}}\big)^{\theta_{1}}\big(\|(a,\u)\|^\ell_{\dot{B}_{2,1}^{\frac n2+1}}\big)^{1-\theta_{1}},
\quad \theta_1=\frac{4}{n-2\sigma+2}\in(0,1),
\end{align*}
this together with Proposition \ref{propagate}  implies that
\begin{align}\label{sa51}
\|(a,\u)\|^\ell_{\dot{B}_{2,1}^{\frac n2+1}}\ge  c_0\big(\|(a,\u)\|^\ell_{\dot{B}_{2,1}^{\frac n2-1}}\big)^{\frac{1}{1-\theta_{1}}}.
\end{align}
For any $ \sigma-1<\frac n2-2,$
\begin{align*}%\label{sa53}
\|\eta\|^\ell_{\dot{B}_{2,1}^{\frac n2-2}}
\le& C \big(\|\eta\|^\ell_{\dot{B}_{2,\infty}^{\sigma-1}}\big)^{\theta_{1}}\big(\|\eta\|^\ell_{\dot{B}_{2,1}^{\frac n2}}\big)^{1-\theta_{1}},
\quad \theta_1=\frac{4}{n-2\sigma+2}\in(0,1),
\end{align*}
which combines with \eqref{sa39} implies
\begin{align}\label{sa54}
\|\eta\|^\ell_{\dot{B}_{2,1}^{\frac n2}}\ge  c_0\big(\|\eta\|^\ell_{\dot{B}_{2,1}^{\frac n2-2}}\big)^{\frac{1}{1-\theta_{1}}}.
\end{align}
Similarly,
for any $ \sigma+1<\frac n2,$
we deduce from interpolation inequality again that
\begin{align*}%\label{sa56}
\|{\mathbb{T}}\|^\ell_{\dot{B}_{2,1}^{\frac n2}}
\le& C \big(\|{\mathbb{T}}\|^\ell_{\dot{B}_{2,\infty}^{\sigma+1}}\big)^{\theta_{1}}\big(\|{\mathbb{T}}\|^\ell_{\dot{B}_{2,1}^{\frac n2+2}}\big)^{1-\theta_{1}},
\end{align*}
from which and \eqref{sa39}, we get
\begin{align}\label{sa57}
\|{\mathbb{T}}\|^\ell_{\dot{B}_{2,1}^{\frac n2+2}}\ge  c_0\big(\|{\mathbb{T}}\|^\ell_{\dot{B}_{2,1}^{\frac n2}}\big)^{\frac{1}{1-\theta_{1}}},
\quad \theta_1=\frac{4}{n-2\sigma+2}\in(0,1).
\end{align}
Now, taking $\beta=1+\theta_{1}>0$ in
\eqref{sa52} and \eqref{sa4}, then combining with   \eqref{sa51}, \eqref{sa54}, \eqref{sa57}, we deduce from \eqref{sa3}
that
\begin{align}\label{sa58}
&\frac{d}{dt}(\|(a,\u)\|^\ell_{\dot{B}_{2,1}^{\frac{n}{2}-1}}
+\|\eta\|^\ell_{\dot{B}_{2,1}^{\frac{n}{2}-2}}
+\|{\mathbb{T}}\|^\ell_{\dot{B}_{2,1}^{\frac{n}{2}}}
+\|(\u,\eta)\|^h_{\dot{B}_{p,1}^{\frac{n}{p}-1}}
+\|(a,{\mathbb{T}})\|^h_{\dot{B}_{p,1}^{\frac{n}{p}}}),
\nonumber\\
&\quad+\widetilde{c}_0(\|(a,\u)\|^\ell_{\dot{B}_{2,1}^{\frac{n}{2}-1}}
+\|\eta\|^\ell_{\dot{B}_{2,1}^{\frac{n}{2}-2}}
+\|{\mathbb{T}}\|^\ell_{\dot{B}_{2,1}^{\frac{n}{2}}}
+\|(\u,\eta)\|^h_{\dot{B}_{p,1}^{\frac{n}{p}-1}}
+\|(a,{\mathbb{T}})\|^h_{\dot{B}_{p,1}^{\frac{n}{p}}})^{1+\frac{4}{n-2\sigma-2}}\le0.
\end{align}
Solving this differential inequality directly, we obtain
\begin{align}\label{sa59}
\|(a,\u)\|^\ell_{\dot{B}_{2,1}^{\frac{n}{2}-1}}
+\|\eta\|^\ell_{\dot{B}_{2,1}^{\frac{n}{2}-2}}
+\|{\mathbb{T}}\|^\ell_{\dot{B}_{2,1}^{\frac{n}{2}}}
+\|(\u,\eta)\|^h_{\dot{B}_{p,1}^{\frac{n}{p}-1}}
+\|(a,{\mathbb{T}})\|^h_{\dot{B}_{p,1}^{\frac{n}{p}}}
\le C(1+t)^{-\frac{n-2\sigma-2}{4}}.
\end{align}
For any $\frac np-\frac n2+\sigma<\gamma_1<\frac np-1,$ by the interpolation inequality we have
\begin{align*}
\|(a,\u)\|^\ell_{\dot{B}_{p,1}^{\gamma_1}}
\le& C\|(a,\u)\|^\ell_{\dot{B}_{2,1}^{\gamma_1+\frac n2-\frac np}}\\
\le&C\big(\|(a,\u)\|^\ell_{\dot{B}_{2,\infty}^{\sigma}}\big)^{\theta_{2}} \big(\|(a,\u)\|^\ell_{\dot{B}_{2,1}^{\frac n2-1}}\big)^{1-\theta_{2}},\quad \theta_{2}=\frac{\frac np -1-\gamma_1}{\frac n2-1-\sigma}\in (0,1),
\end{align*}
which combines  with  Proposition \ref{propagate} gives
\begin{align}\label{sa60}
\|(a,\u)\|^\ell_{\dot{B}_{p,1}^{\gamma_1}}
\le C(1+t)^{-\frac{(\frac n2-\sigma-1)(1-\theta_{2})}{2}}=C(1+t)^{-\frac{n}{2}(\frac 12-\frac 1p)-\frac{\gamma_1-\sigma}{2}}.
\end{align}
In  the light of
$\frac np-\frac n2+\sigma<\gamma_1<\frac np-1,$
 we see that
$$\|(a^h,\u^h)\|_{\dot{B}_{p,1}^{\gamma_1}}\le C(\|a\|^h_{\dot{B}_{p,1}^{\frac np}}+\|\u\|^h_{\dot{B}_{p,1}^{\frac np-1}})\le C(1+t)^{-\frac{n-2\sigma-2}{4}},
$$
from which and \eqref{sa60} gives
\begin{align*}
\|(a,\u)\|_{\dot{B}_{p,1}^{\gamma_1}}
\le&C(\|(a,\u)\|^\ell_{\dot{B}_{p,1}^{\gamma_1}}+\|(a,\u)\|^h_{\dot{B}_{p,1}^{\gamma_1}})\nonumber\\
\le& C(1+t)^{-\frac{n}{2}(\frac 12-\frac 1p)-\frac{\gamma_1-\sigma}{2}}+C(1+t)^{-\frac{d-2\sigma-2}{4}}\nonumber\\
\le& C(1+t)^{-\frac{n}{2}(\frac 12-\frac 1p)-\frac{\gamma_1-\sigma}{2}}.
\end{align*}

Hence,
thanks to the embedding relation
$\dot{B}^{0}_{p,1}(\R^n)\hookrightarrow L^p(\R^n)$, one infer that
\begin{align*}
\|\Lambda^{\gamma_1} (a,\u)\|_{L^p}
\le& C(1+t)^{-\frac{n}{2}(\frac 12-\frac 1p)-\frac{\gamma_1-\sigma}{2}}.
\end{align*}
For any $\frac np-\frac n2-1+\sigma<\gamma_2<\frac np-1,$  we have
\begin{align*}
\|\eta\|^\ell_{\dot{B}_{p,1}^{\gamma_2}}
\le& C\|\eta\|^\ell_{\dot{B}_{2,1}^{\gamma_2+\frac n2-\frac np}}
\le C\big(\|\eta\|^\ell_{\dot{B}_{2,\infty}^{\sigma-1}}\big)^{\theta_{3}} \big(\|\eta\|^\ell_{\dot{B}_{2,1}^{\frac n2-2}}\big)^{1-\theta_{3}},\quad \theta_{3}=\frac{\frac np -\gamma_2-2}{\frac n2-1-\sigma}\in (0,1),
\end{align*}
which  gives
\begin{align}\label{sa61}
\|\eta\|^\ell_{\dot{B}_{p,1}^{\gamma_2}}
\le C(1+t)^{-\frac{n}{2}(\frac 12-\frac 1p)-\frac{\gamma_2-\sigma+1}{2}}.
\end{align}
As
$\gamma_2<\frac np-1,$
 we see  in the high frequency part that
\begin{align}\label{sa61+61}
\|\eta\|^h_{\dot{B}_{p,1}^{\gamma_2}}\le C\|\eta\|^h_{\dot{B}_{p,1}^{\frac np-1}}\le C(1+t)^{-\frac{n-2\sigma-2}{4}}.
\end{align}
From \eqref{sa61} and \eqref{sa61+61}, we have
\begin{align*}
\|\eta\|_{\dot{B}_{p,1}^{\gamma_2}}
\le C(\|\eta\|^\ell_{\dot{B}_{p,1}^{\gamma_2}}+\|\eta\|^h_{\dot{B}_{p,1}^{\gamma_2}})
\le C(1+t)^{-\frac{n}{2}(\frac 12-\frac 1p)-\frac{\gamma_2-\sigma+1}{2}}
\end{align*}
which further implies for any $\frac np-\frac n2-1+\sigma<\gamma_2<\frac np-1,$  that
\begin{align*}
\|\Lambda^{\gamma_2}\eta\|_{L^p}
\le C(1+t)^{-\frac{n}{2}(\frac 12-\frac 1p)-\frac{\gamma_2-\sigma+1}{2}}.
\end{align*}

For any $\frac np-\frac n2+1+\sigma<\gamma_3<\frac np,$ by the interpolation inequality we have
\begin{align*}
\|{\mathbb{T}}\|^\ell_{\dot{B}_{p,1}^{\gamma_3}}
\le& C\|{\mathbb{T}}\|^\ell_{\dot{B}_{2,1}^{\gamma_3+\frac n2-\frac np}}
\le C\big(\|{\mathbb{T}}\|^\ell_{\dot{B}_{2,\infty}^{\sigma+1}}\big)^{\theta_{4}} \big(\|{\mathbb{T}}\|^\ell_{\dot{B}_{2,1}^{\frac n2}}\big)^{1-\theta_{4}},\quad \theta_{4}=\frac{\frac np -\gamma_3}{\frac n2-1-\sigma}\in (0,1),
\end{align*}
which  gives
\begin{align}\label{sa61+3}
\|{\mathbb{T}}\|^\ell_{\dot{B}_{p,1}^{\gamma_3}}
\le C(1+t)^{-\frac{(\frac n2-\sigma-1)(1-\theta_{4})}{2}}=C(1+t)^{-\frac{n}{2}(\frac 12-\frac 1p)-\frac{\gamma_3-\sigma-1}{2}}.
\end{align}
In  the light of
$\frac np-\frac n2+1+\sigma<\gamma_3<\frac np,$
 we see that
\begin{align}\label{sa61+61+3}
\|{\mathbb{T}}\|^h_{\dot{B}_{p,1}^{\gamma_3}}\le C\|{\mathbb{T}}\|^h_{\dot{B}_{p,1}^{\frac np}}\le C(1+t)^{-\frac{n-2\sigma-2}{4}}.
\end{align}
From \eqref{sa61+3} and \eqref{sa61+61+3}, we have
\begin{align*}
\|{\mathbb{T}}\|_{\dot{B}_{p,1}^{\gamma_3}}
\le C(\|{\mathbb{T}}\|^\ell_{\dot{B}_{p,1}^{\gamma_3}}+\|{\mathbb{T}}\|^h_{\dot{B}_{p,1}^{\gamma_3}})
\le C(1+t)^{-\frac{n}{2}(\frac 12-\frac 1p)-\frac{\gamma_3-\sigma-1}{2}},
\end{align*}
which further implies for any $\frac np-\frac n2+1+\sigma<\gamma_3<\frac np,$  that
\begin{align*}
\|\Lambda^{\gamma_3}{\mathbb{T}}\|_{L^p}
\le C(1+t)^{-\frac{n}{2}(\frac 12-\frac 1p)-\frac{\gamma_3-\sigma-1}{2}}.
\end{align*}
This complete the proof of the Theorem \ref{dingli3}. $\hspace{8.2cm} \square$

\vskip .3in
\section*{\bf Acknowledgments}
This work is  supported by  National Natural Science Foundation of China under grant numbers 12001377,11601533, 11971356,
National Natural Science Foundation key project of China under grant number 11831003,
 and Natural Science Foundation of Guangdong Province of China under grant numbers 2018A030313024, 2020B1515310008, and Project of Educational Commission of Guangdong Province of China under grant number 2019KZDZX1007.

\section*{Appendix: The proof of our claim \eqref{key0}--\eqref{key}}

\setcounter{theorem}{0}
\setcounter{equation}{0}
\renewcommand{\theremark}{A.\arabic{remark}}
\renewcommand{\theequation}{A.\arabic{equation}}
\renewcommand{\thetheorem}{A.\arabic{theorem}}

\begin{proof}
We first use  Bony's decomposition to rewrite
\begin{align}\label{mingzai}
fg=\dot{T}_fg+\dot{T}_gf+\dot{R}(f,g).
\end{align}
Due to  H\"older's inequality, Bernstein's inequality and the embedding $\dot{B}_{p,1}^{\frac{n}{p}}(\R^n)\hookrightarrow L^\infty(\R^n)$, we can get
\begin{align}\label{ty1}
\|\dot{\Delta}_j(\dot{T}_fg)\|_{L^2}&\lesssim\sum_{|k-j|\leq 4}\|\dot{\Delta}_j(\dot{S}_{k-1}f\dot{\Delta}_kg)\|_{L^2}
\nonumber\\
&\lesssim\sum_{|k-j|\leq 4}\|\dot{S}_{k-1}f\|_{L^\infty}\|\dot{\Delta}_kg\|_{L^2}\nonumber\\
&\lesssim\sum_{|k-j|\leq 4}\|f\|_{L^\infty}\|\dot{\Delta}_kg\|_{L^2}\nonumber\\
&\lesssim 2^{-\sigma j}\|f\|_{\dot{B}_{p,1}^{\frac{n}{p}}}\|g\|_{\dot{B}_{2,\infty}^{\sigma}}.
\end{align}
Similarly,  for any $\sigma<\frac np$, the second term in  \eqref{mingzai} can be estimated  as follows
\begin{align}\label{ty2}
\|\dot{\Delta}_j(\dot{T}_gf)\|_{L^2}
&\lesssim\sum_{|k-j|\leq 4}\sum_{k^\prime\leq k-2}\|\dot{\Delta}_{k^\prime}g\|_{L^{2p/(p-2)}}\|\dot{\Delta}_kf\|_{L^p}\nonumber\\
&\lesssim\sum_{|k-j|\leq 4}(\sum_{k^\prime\leq k-2}2^{k^\prime(\frac{n}{p}-\sigma)}2^{k^\prime\sigma}
\|\dot{\Delta}_{k^\prime}g\|_{L^2})\|\dot{\Delta}_kf\|_{L^p}\nonumber\\
&\lesssim 2^{-\sigma j}\|f\|_{\dot{B}_{p,1}^{\frac{n}{p}}}\|g\|_{\dot{B}_{2,\infty}^{\sigma}}.
\end{align}
In view of the fact that $\sigma\ge-\frac{n}{p}$, we can deal with the last term in \eqref{mingzai}
\begin{align}\label{ty3}
\|\dot{\Delta}_j\dot{R}(f,g)\|_{L^2}&\lesssim \sum_{k\geq j-3}\sum_{|k-k^\prime|\leq 1}
\|\dot{\Delta}_j(\dot{\Delta}_{k}f\dot{\Delta}_{k^\prime}g)\|_{L^2}\nonumber\\
&\lesssim 2^{\frac{n}{p}j}\sum_{k\geq j-3}\sum_{|k-k^\prime|\leq 1}\|\dot{\Delta}_{k}f\dot{\Delta}_{k^\prime}g\|_{L^{\frac{2p}{p+2}}}\nonumber\\
&\lesssim 2^{\frac{n}{p}j}\sum_{k\geq j-3}\sum_{|k-k^\prime|\leq 1}2^{-k\frac{n}{p}}2^{k\frac{n}{p}}\|\dot{\Delta}_{k}f\|_{L^{p}}2^{-k^\prime\sigma}
2^{k^\prime\sigma}\|\dot{\Delta}_{k^\prime}g\|_{L^{2}}\nonumber\\
&\lesssim2^{\frac{n}{p}j}\sum_{k\geq j-3}2^{-k(\sigma+\frac{n}{p})}d_k\|f\|_{\dot{B}_{p,1}^{\frac{n}{p}}}
\|g\|_{\dot{B}_{2,\infty}^{\sigma}}\nonumber\\
&\lesssim2^{-\sigma j}\|f\|_{\dot{B}_{p,1}^{\frac{n}{p}}}\|g\|_{\dot{B}_{2,\infty}^{\sigma}}.
\end{align}
The combination of  \eqref{ty1}--\eqref{ty3} gives  \eqref{key0}.

Now, we are in a position to prove \eqref{key1}. We also use Bony's decomposition to write
$
fg^\ell=\dot{T}_f{g^\ell}+\dot{T}_{g^\ell} f+\dot{R}(f,g^\ell).
$
It follows from the H\"older inequality and the  Bernstein inequality
 to get
\begin{align}\label{ty5}
\|\dot{\Delta}_j(\dot{T}_fg^\ell)\|_{L^2}
&\lesssim\sum_{|k-j|\leq 4}\sum_{k^\prime\leq k-2}\|\dot{\Delta}_{k^\prime}f\|_{L^{\infty}}\|\dot{\Delta}_kg^\ell\|_{L^2}\nonumber\\
&\lesssim\sum_{|k-j|\leq 4}\sum_{k^\prime\leq k-2}2^{\frac{nk^\prime}{p}}
\|\dot{\Delta}_{k^\prime}f\|_{L^p}\|\dot{\Delta}_kg^\ell\|_{L^2}\nonumber\\
&\lesssim\sum_{|k-j|\leq 4}(\sum_{k^\prime\leq k-2}2^{k^\prime}2^{k^\prime(\frac{n}{p}-1)}
\|\dot{\Delta}_{k^\prime}f\|_{L^p})
2^{-k(\sigma+1)}2^{k(\sigma+1)}
\|\dot{\Delta}_kg^\ell\|_{L^2}\nonumber\\
&\lesssim 2^{- \sigma j}\|f\|_{\dot{B}_{p,1}^{\frac{n}{p}-1}}\|g^\ell\|_{\dot{B}_{2,\infty}^{\sigma+1}}.
\end{align}
Along   the same lines, for any $\sigma<\frac n2-1$, one has
\begin{align}\label{ty6}
\|\dot{\Delta}_j(\dot{T}_{g^\ell}f)\|_{L^2}&\leq \sum_{|k-j|\leq 4}\sum_{k^\prime\leq k-2}\|\dot{\Delta}_j(\dot{\Delta}_{k^\prime}{g^\ell}\dot{\Delta}_kf)\|_{L^2}\nonumber\\
&\leq\sum_{|k-j|\leq 4}\sum_{k^\prime\leq k-2}\|\dot{\Delta}_{k^\prime}{g^\ell}\|_{L^{{\frac{2p}{p-2}}}}\|\dot{\Delta}_kf\|_{L^p}\nonumber\\
&\lesssim\sum_{|k-j|\leq 4}\sum_{k^\prime\leq k-2}2^{k^\prime(\frac{2n}{p}-\frac{n}{2})}\|\dot{\Delta}_{k^\prime}{g^\ell}\|_{L^p}\|\dot{\Delta}_kf\|_{L^p}\nonumber\\
&\lesssim\sum_{|k-j|\leq 4}(\sum_{k^\prime\leq k-2}2^{-k^\prime(\sigma-\frac n2+1)}2^{k^\prime(\sigma+\frac{2n}{p}-n+1)}
\|\dot{\Delta}_{k^\prime}{g^\ell}\|_{L^p})
\|\dot{\Delta}_kf\|_{L^p}\nonumber\\
&\lesssim 2^{-(\sigma+\frac{n}{p}-\frac{n}{2}) j}\|f\|_{\dot{B}_{p,1}^{\frac{n}{p}-1}}\|g^\ell\|_{\dot{B}_{p,\infty}^{\sigma+\frac{2n}{p}-n+1}}
\nonumber\\
&\lesssim 2^{-(\sigma+\frac{n}{p}-\frac{n}{2}) j}\|f\|_{\dot{B}_{p,1}^{\frac{n}{p}-1}}\|g^\ell\|_{\dot{B}_{2,\infty}^{\sigma+1}}.
\end{align}
As $\sigma+\frac{2n}{p}-\frac{n}{2}>0,$ we can deal with the
the remainder term as follows:
\begin{align}\label{ty7}
\|\dot{\Delta}_j\dot{R}(f,g^\ell)\|_{L^2}&\lesssim \sum_{k\geq j-3}\sum_{|k-k^\prime|\leq 1}
\|\dot{\Delta}_j(\dot{\Delta}_{k}f\dot{\Delta}_{k^\prime}g^\ell)\|_{L^2}\nonumber\\
&\lesssim 2^{j(\frac{2n}{p}-\frac{n}{2})}\sum_{k\geq j-3}\sum_{|k-k^\prime|\leq 1}\|\dot{\Delta}_{k}f\|_{L^{p}}
2^{(\frac{n}{2}-\frac{n}{p}k^\prime)}\|\dot{\Delta}_{k^\prime}g^\ell\|_{L^{2}}\nonumber\\
&\lesssim2^{j(\frac{2n}{p}-\frac{n}{2})}\sum_{k\geq j-3}2^{-k(\sigma+\frac{2n}{p}-\frac{n}{2})}d_k\|f\|_{\dot{B}_{p,1}^{\frac{n}{p}-1}}
\|g^\ell\|_{\dot{B}_{2,\infty}^{\sigma+1}}\nonumber\\
&
\lesssim2^{-\sigma j}\|f\|_{\dot{B}_{p,1}^{\frac{n}{p}-1}}\|g^\ell\|_{\dot{B}_{2,\infty}^{\sigma+1}}.
\end{align}
Summing up \eqref{ty5}--\eqref{ty7}, noticing $\sigma+\frac{n}{p}-\frac{n}{2}\le\sigma$, we  can obtain \eqref{key1}.

Finally, we  are concerned with the  proof of  \eqref{key}. Thanks to Bony's decomposition again
\begin{align}\label{mingzai3}
fg^h=\dot{T}_f{g^h}+\dot{T}_{g^h} f+\dot{R}(f,g^h).
\end{align}
By virtue of  the H\"older inequality, Bernstein's  inequality and the fact $$\frac{n}{2}-1\le\frac{n}{p}, \quad\hbox{and}\quad\sigma<\frac n2-1,$$ one deduces that
\begin{align}\label{ty8}
\|\dot{\Delta}_j(\dot{T}_fg^h)\|_{L^2}
&\lesssim\sum_{|k-j|\leq 4}(\sum_{k^\prime\leq k-2}\|\dot{\Delta}_{k^\prime}f\|_{L^{{\frac{2p}{p-2}}}})\|\dot{\Delta}_kg^h\|_{L^p}\nonumber\\
&\lesssim\sum_{|k-j|\leq 4}(\sum_{k^\prime\leq k-2}2^{k^\prime(\frac{2n}{p}-\frac{n}{2})}
\|\dot{\Delta}_{k^\prime}f\|_{L^p})\|\dot{\Delta}_kg^h\|_{L^p}\nonumber\\
&\lesssim\sum_{|k-j|\leq 4}(\sum_{k^\prime\leq k-2}2^{k^\prime(\frac{n}{p}-\frac{n}{2}+1)}2^{k^\prime(\frac{n}{p}-1)}
\|\dot{\Delta}_{k^\prime}f\|_{L^p})
2^{-k\sigma}2^{k(\sigma+\frac{n}{p}-\frac{n}{2}+1)}\|\dot{\Delta}_kg^h\|_{L^p}\nonumber\\
&\lesssim 2^{-\sigma j}\|f\|_{\dot{B}_{p,1}^{\frac{n}{p}-1}}\|g^h\|_{\dot{B}_{p,\infty}^{\sigma+\frac{n}{p}-\frac{n}{2}+1}}\nonumber\\
&\lesssim 2^{-\sigma j}\|f\|_{\dot{B}_{p,1}^{\frac{n}{p}-1}}\|g^h\|_{\dot{B}_{p,1}^{\frac{n}{p}}}.
\end{align}
The term
$\|\dot{\Delta}_j\dot{R}(f,g^h)\|_{L^2}$ can be dealt in  a similar manner,
\begin{align}\label{ty9}
\|\dot{\Delta}_j\dot{R}(f,g^h)\|_{L^2}\lesssim 2^{-\sigma j}\|f\|_{\dot{B}_{p,1}^{\frac{n}{p}-1}}\|g^h\|_{\dot{B}_{p,1}^{\frac{n}{p}}}.
\end{align}
Thanks to Bernstein's  inequality, we have
\begin{align}\label{ty10}
\|\dot{\Delta}_j(\dot{T}_{g^h}f)\|_{L^2}&\lesssim \sum_{|k-j|\leq 4}\sum_{k^\prime\leq k-2}\|\dot{\Delta}_j(\dot{\Delta}_{k^\prime}{g^h}\dot{\Delta}_kf)\|_{L^2}\lesssim\sum_{|k-j|\leq 4}(\sum_{k^\prime\leq k-2}\|\dot{\Delta}_{k^\prime}{g^h}\|_{L^{{\frac{2p}{p-2}}}})\|\dot{\Delta}_kf\|_{L^p}\nonumber\\
&\lesssim\sum_{|k-j|\leq 4}(\sum_{k^\prime\leq k-2}2^{k^\prime(\frac{2n}{p}-\frac{n}{2})}\|\dot{\Delta}_{k^\prime}{g^h}\|_{L^p})\|\dot{\Delta}_kf\|_{L^p}\nonumber\\
&\lesssim\sum_{|k-j|\leq 4}(\sum_{k^\prime\leq k-2}2^{-k^\prime(\sigma+1-\frac n2)}2^{k^\prime(\sigma+\frac{2n}{p}-n+1)}
\|\dot{\Delta}_{k^\prime}{g^h}\|_{L^p})
\|\dot{\Delta}_kf\|_{L^p}\nonumber\\
&\lesssim 2^{-(\sigma-\frac n2+\frac np) j}\|f\|_{\dot{B}_{p,1}^{\frac{n}{p}-1}}\|g^h\|_{\dot{B}_{p,\infty}^{\sigma+\frac{2n}{p}-n+1}},
\end{align}
from which  and relations $$\frac np-\frac n2\le 0,  \quad \sigma+\frac{2n}{p}-n+1\le \frac np,$$ we have
\begin{align}\label{ty11}
\|\dot{T}_{g^h}f\|^\ell_{\dot{B}_{2,1}^{\sigma}}&\lesssim \|\dot{T}_{g^h}f\|^\ell_{\dot{B}_{2,1}^{\sigma-\frac n2+\frac np}}\lesssim
\|f\|_{\dot{B}_{p,1}^{\frac{n}{p}-1}}\|g^h\|_{\dot{B}_{p,\infty}^{\sigma+\frac{2n}{p}-n+1}}\lesssim
\|f\|_{\dot{B}_{p,1}^{\frac{n}{p}-1}}\|g^h\|_{\dot{B}_{p,1}^{\frac{n}{p}}}.
\end{align}
Collecting \eqref{ty8}, \eqref{ty9} and\eqref{ty11}, the desired estimate \eqref{key} is obtained.

Consequently, we conclude the proof of \eqref{key0}--\eqref{key}.
\end{proof}

%\noindent {\large\bf  References}

\end{document}